\documentclass{scrartcl}

\usepackage{mixed}
\author{Yuki Irie\thanks{Yuki Irie \newline Research Alliance Center for Mathematical Sciences, Tohoku University \newline yirie@tohoku.ac.jp}\ \thanks{This work was partially carried out at Chiba University.}}
\date{}
\title{Mixed-Radix Nim}
\hypersetup{
 pdfauthor={Yuki Irie},
 pdftitle={Mixed-Radix Nim},
 pdfkeywords={},
 pdfsubject={},
 pdfcreator={Emacs 25.2.1 (Org mode 8.3.4)}, 
 pdflang={English}}
\begin{document}

\maketitle
We present take-away games whose Sprague-Grundy functions are given by the Nim sum of heap sizes
in a mixed base \(\beta\). Let \(\Delta_\beta\) be the set of such games.
We give a necessary and sufficient condition for the existence of a minimum of \(\Delta_\beta\),
and a recursive construction of the maximum of \(\Delta_\beta\).

\section{Introduction}
\label{sec:orgheadline30}
\comment{Intro}
\label{sec:orgheadline1}
Nim is a game played with heaps of coins.
Sprague \cite{sprague-uber-1935} and Grundy \cite{grundy-mathematics-1939} 
independently found that every (short impartial) game can be solved using a nonnegative integer-valued function,
which is called the \emph{Sprague-Grundy function} of the game.
They also proved that the Sprague-Grundy function of Nim is
\[
 \nsigma[2]((x^0, \ldots, x^{m - 1})) = x^0 \oplus_2 \cdots \oplus_2 x^{m - 1},
\]
where \(x^i\) is the size of the \(i\)th heap and \(\oplus_2\) is addition without carry in base 2.

Roughly speaking, the questions of this paper are as follows:
\begin{enumerate}
\item Given a function \(\phi\), is there a nice game with Sprague-Grundy function \(\phi\)?
\item If so, can we characterize such games?
\end{enumerate}
We will address these questions in the case when \(\phi\) is a generalization of \(\nsigma[2]\).

Some generalizations of \(\nsigma[2]\) are known.
For an integer \(b\) greater than 1,
Flanigan \cite{flanigan-Nim-1980} presented a take-away game
whose Sprague-Grundy function is 
\[
  \nsigma[b]((x^0, \ldots, x^{m - 1})) = x^0 \oplus_b \cdots \oplus_b x^{m - 1},
\]
where \(\oplus_b\) is addition without carry in base \(b\).
Fraenkel and Lorberbom \cite{fraenkel-Nimhoff-1991} constructed
take-away games whose Sprague-Grundy functions are
\[
  \nsigma[[b]]((x^0, \ldots, x^{m - 1})) = b \left( \left \lfloor \frac{x^0}{b} \right \rfloor \oplus_2 \cdots \oplus_2 \left \lfloor \frac{x^{m - 1}}{b} \right \rfloor \right) + ((x^0 + \cdots + x^{m - 1}) \bmod b),
\]
where \(n \bmod b\) is the nonnegative remainder of \(n\) divided by \(b\). 
They also presented the \emph{minimum} of take-away games with Sprague-Grundy function \(\nsigma[[b]]\). 
Blass, Fraenkel, and Guelman \cite{blass-how-1998}  determined
the \emph{maximum} of take-away games with Sprague-Grundy function \(\nsigma[2]\).

In this paper, we generalize the above results by using mixed-radix number systems.
For a mixed base \(\ybs\) (see Subsection \ref{orgtarget1}),
we present take-away games whose Sprague-Grundy functions are 
\[
  \nsigma[\ybs]((x^0, \ldots, x^{m - 1})) = x^0 \oplus_{\ybs} \cdots \oplus_{\ybs} x^{m - 1},
\]
where \(\oplus_{\ybs}\) is addition without carry in base \(\ybs\) (Theorem \ref{orgtarget2}).
In general, \emph{minimal} elements of such games may be not unique.
We prove that there is a minimum of such games if and only if \(\nsigma[\ybs] = \nsigma[[b]]\) for some \(b\) (Theorem \ref{orgtarget3}).
We also give a recursive construction of the maximum of such games (Theorem \ref{orgtarget4}).

\subsection{Take-away games and Sprague-Grundy systems}
\label{sec:orgheadline8}
To describe our results precisely, we introduce some notation.

We first define take-away games as the following digraphs.
Let \(\cP\) be a subset \footnote{The set \(\cP\) is \(\NN^m\) except for Examples \ref{orgtarget5} and \ref{orgtarget6}.} of \(\NN^m\), where \(\NN\) is the set of nonnegative integers and \(m \in \NN\).
For a subset \(\sgC\) of \(\NN^m \setminus \set{(0,\ldots,0)}\), 
let \(\Gamma(\cP, \sgC)\) be the digraph with vertex set \(\cP\) and edge set 
\[
 \set{(X, Y) \in \cP^2 : X - Y \in \sgC}.
\]
Then \(\Gamma(\cP, \sgC)\) is called a \emph{take-away game}. 

 \begin{example}[Nim]
 \comment{Exm. [Nim]}
\label{sec:orgheadline2}
\label{orgtarget7}
Let \(\cP = \NN^m\).
Consider
\begin{equation}
\label{orgtarget8}
 \sgCord = \Set{\ccC \in \NN^m : \wt(C) = 1},
\end{equation}
where \(\wt(C)\) is the Hamming weight of \(C\), that is,
the number of nonzero components of \(C\).
The game \(\Gamma(\NN^m, \sgCord)\) is called \emph{Nim}. 
As we have mentioned above, the Sprague-Grundy function of Nim is given by
\[
 \nsigma[2](X) = \nsigma[2, m](X) = x^0 \oplus_2 \cdots \oplus_2 x^{m - 1},
\]
where \(x^i\) denotes the \(i\)th component of \(X\), that is, \(X = (x^0, \ldots, x^{m - 1})\).
 
\end{example}

\comment{connect}
\label{sec:orgheadline3}
As in Example \ref{orgtarget7}, elements in \(\NN^m\) will be denoted by capital Latin letters, and
components of them by small Latin letters. For example, if \(C \in \NN^m\), then \(C = (c^0, \ldots, c^{m - 1})\).

\comment{SG system}
\label{sec:orgheadline4}
Because the set \(\sgC\) plays a key role in this paper,
we introduce the following notation.
Let \(\sg_{\Gamma(\cP, \sgC)}\) denote the Sprague-Grundy function of a take-away game \(\Gamma(\cP, \sgC)\).
For a nonnegative integer-valued function \(\phi: \cP \to \NN\),
the set \(\sgC\) is called a \emph{Sprague-Grundy system} of \(\phi\) if \(\sg_{\Gamma(\cP, \sgC)}\) equals \(\phi\):
\[
 \sg_{\Gamma(\cP, \sgC)}(X) = \phi(X) \tforevery X \in \cP.
\]
For example, let \(\sgCord[2,m][1]\) be the right-hand side of (\ref{orgtarget8}).
Then \(\sgCord[2,m][1]\) is a Sprague-Grundy system of the function \(\nsigma[2,m]\).
Let \(\Delta(\phi)\) denote the set of Sprague-Grundy systems of \(\phi\).
By the definition of Sprague-Grundy functions,
\begin{equation}
\label{orgtarget9}
\text{if\ } \sgC, \sgE \in \Delta(\phi) \tand \sgC \subseteq \sgD \subseteq \sgE, \text{\ then\ } \sgD \in \Delta(\phi) 
\end{equation}
(see Section \ref{orgtarget10}).

\comment{Questions}
\label{sec:orgheadline5}
We address the following three questions of Sprague-Grundy systems:
\begin{description}
\item[{1.}] Is there a Sprague-Grundy system of \(\phi\)? \par
We will prove that \(\nsigma[\ybs]\) has a Sprague-Grundy system. In other words,
\(\nsigma[\ybs]\) can be realized as the Sprague-Grundy function of a game.
\item[{2.}] What are \emph{minimal systems} of \(\phi\)? \par
A minimal (with respect to inclusion) element of \(\Delta(\phi)\) will be called a \emph{minimal system} of \(\phi\).
In general, \(\phi\) has multiple minimal systems.
We will determine when \(\nsigma[\ybs]\) has a unique minimal system, which will be called the \emph{minimum system} of \(\phi\).
We also investigate the structure of minimal systems of \(\nsigma[\ybs]\) for an arbitrary mixed base \(\ybs\).
\item[{3.}] What is the \emph{maximum system} of \(\phi\)? \par
If \(\Delta(\phi) \neq \emptyset\), then \(\Delta(\phi)\) has a maximum element, which will be called the \emph{maximum system} of \(\phi\). 
We will give a recursive construction of the maximum system of \(\nsigma[\ybs]\).
\end{description}

In this paper, we will restrict our attention to the case \(\cP = \NN^m\). 
We here give, however, two examples of \(\phi:\cP \to \NN\) with \(\cP \subsetneq \NN^m\).

 \begin{example}[Sprague-Grundy function of a \(b\)-saturation of mis\`{e}re Nim \cite{irie-misere-2018}]
 \comment{Exm. [Sprague-Grundy function of a \(b\)-saturation of mis\`{e}re Nim \cite{irie-misere-2018}]}
\label{sec:orgheadline6}
\label{orgtarget5}
Let \(b\) be an integer greater than 1.
For an integer \(n\), let \(\ord_b(n)\) be the \(b\)-adic order of \(n\), that is,
\[
 \ord_b(n) = \begin{cases}
 \max \Set{L \in \NN : b^L \text{\ divides\ } n } & \tif  n \neq 0,\\
 \infty & \tif n = 0.
 \end{cases}
\]
For example, \(\ord_2(1) = 0\) and \(\ord_2(12) = 2\).
Let \(\cP = \NN^m \setminus \set{(0,\ldots, 0)}\). 
The game \(\Gamma(\cP, \sgCord[2,m][1])\) is called \emph{mis\`{e}re Nim}. \footnote{No explicit formula for the Sprague-Grundy function of mis\`{e}re Nim is known.} 
Define a function \(\phi^b : \cP \to \NN\) by
\[
 \phi^b(X) = \nsigma[b](X) \oplus_b (\ybsp[1 + \mord_b(X)][b] - 1),
\]
where \(\mord_{b}(X) = \min \set{\ord_{b}(\ccx^i) : i \in \Omega}\)
and \(\Omega = \set{0,1,\ldots, m - 1}\). For example, 
\[
 \mord_2((0,1)) = \min \set{\ord_2(0), \ord_2(1)} = \min \set{\infty, 0} = 0,
\]
so
\[
 \phi^2((0,1)) = 0 \oplus_2 1 \oplus_2 (2^{1 + 0} - 1) = 0.
\]
Consider
\[
 \sgCord[b] =  \sgCord[b, m] = \Set{\ccC \in \NN^m \setminus \set{(0, \ldots, 0)}: \ord_{b}\left(\sum_{i \in \Omega} \ccc^i\right) = \mord_{b}(C)}.
\]
For example, \((1, 2) \in \sgCord[2]\) and \((1,1) \not \in \sgCord[2]\) because
\(\ord_2(1 + 2) = 0 = \mord_2((1,2))\) and \(\ord_2(1 + 1) = 1 > 0 = \mord_2((1,1))\).
Then \(\sgCord[b]\) is a Sprague-Grundy system of \(\phi^b\).
In other words, the Sprague-Grundy function of the game \(\Gamma(\cP, \sgCord[b])\) can be written explicitly
as follows: \(\sg_{\Gamma(\cP, \sgCord[b])} = \phi^b\).
 
\end{example}

 \begin{example}[Sprague-Grundy function of a \(b\)-saturation of Welter's game \cite{welter-theory-1954}, \cite{sato-game-1968, sato-mathematical-1970, sato-maya-1970}, \cite{irie-psaturations-2018}]
 \comment{Exm. [Sprague-Grundy function of a \(b\)-saturation of Welter's game \cite{welter-theory-1954}, \cite{sato-game-1968, sato-mathematical-1970, sato-maya-1970}, \cite{irie-psaturations-2018}]}
\label{sec:orgheadline7}
\label{orgtarget6}
Let \(b\) be an integer greater than 1.
Set
\[
 \cP = \set{X \in \NN^m : x^i \neq x^j \text{\ whenever\ } i \neq j}.
\]
Define a function \(\phi^b : \cP \to \NN\) by
\[
 \phi^b(X) = \nsigma[b](X) \oplus_b \bigoplusp[i < j][][b] \left(\ybsp[1 + \ord_b(x^i - x^j)][b] - 1\right),
\]
Then \(\sgCord[b]\) is a Sprague-Grundy system of \(\phi^b\).
It is well known that \(\sgCord[2,m][1]\) is also a Sprague-Grundy system of \(\phi^2\),
since \(\Gamma(\cP, \sgCord[2,m][1])\) is the ordinary Welter's game.
 
\end{example}

\subsection{Nim sum in a mixed base}
\label{sec:orgheadline14}
\label{orgtarget1}
We define \(\nsigma[\ybs]\) and present its Sprague-Grundy system.

\comment{Notation}
\label{sec:orgheadline9}
\label{orgtarget11}
We first introduce some notation.
Throughout this paper, \(\ybs\) denotes a sequence \((\ybs_L)_{L \in \NN} \in \NN^\NN\) with \(\ybs[L] \ge 2\) for every \(L \in \NN\).
Let \(\range{\ybs[L]}\) denote \(\set{0,1,\ldots, \ybs[L] - 1}\) equipped with the following two operations:
for \(a, b \in \range{\ybs[L]}\),
\[
 a \oplus b = (a + b) \bmod \ybs[L]  \quad \tand \quad a \ominus b = (a - b) \bmod \ybs[L].
\]
For example, in \(\range{3}\), \(1 \oplus 2 = 0\) and \(1 \ominus 2 = 2\).
Let \(\ybsp[L] = \ybs[0] \cdot \ybs[1] \cdots \ybs[L - 1]\).
For \(n \in \NN\), let \(n_L^{\ybs}\) denote the \(L\)th digit in the mixed base \(\ybs\) expansion of \(n\), that is,
if \(n_{\ge L}^{\ybs}\) is the quotient of \(n\) divided by \(\ybsp[L]\), then \(n_L^{\ybs} = n_{\ge L}^{\ybs} \bmod \ybs[L]\).
We will think of \(n_L^{\ybs}\) as an element of \(\range{\ybs[L]}\).
By definition,
\[
 n = \sum_{L \in \NN} n_L^{\ybs} \ybsp[L].
\]
Of course, if \(\ybs[L] = \ybs[0]\) for every \(L \in \NN\), then \(n_L^{\ybs}\) is the \(L\)th digit in the ordinary base \(\ybs[0]\) expansion of \(n\).
When no confusion can arise, we write \(n_L\) and \(n_{\ge L}\) instead of \(n_L^{\ybs}\) and \(n_{\ge L}^{\ybs}\), respectively. 
Let \(\Nb\) denote \(\NN\) equipped with the following two operations: 
for \(n, h \in \Nb\),
\[
 n \oplus h = \sum_{L \in \NN} (n_L \oplus h_L) \ybsp[L] \quad \tand \quad n \ominus h = \sum_{L \in \NN} (n_L \ominus h_L) \ybsp[L].
\]
It is convenient to represent \(n \in \Nb\) by \([n_0, n_1, \ldots, n_N]\) when \(n_L = 0\) for every \(L > N\).
For \(X \in \Nb^m\), define
\begin{equation}
\label{orgtarget12}
 \nsigma[\ybs](X) = \nsigma[\ybs, m](X) = x^0 \oplus \cdots \oplus x^{m - 1}.
\end{equation}
Let \(\nsigma[\ybs]<L>(X) = (\nsigma[\ybs](X))_L\) (\(=\) the \(L\)th digit of \(\nsigma[\ybs](X)\)) for \(L \in \NN\).

 \begin{example}
 \comment{Exm. ybs}
\label{sec:orgheadline10}
\label{orgtarget13}
Let \(\ybs = (60, 24, 7, \ldots)\) and \(x^0 = [30,5,1] \in \Nb\).
Then
\[
 \bnum{30, 5, 1} = 30 + 5 \cdot 60 + 1 \cdot 60 \cdot 24 = 1770.
\]
Let \(x^1 = \bnum{40, 15, 6} \in \Nb\) and \(X = (x^0, x^1) \in \Nb^2\). Then
\[
 \nsigma[\ybs](X) = \bnum{30 \oplus 40, 5 \oplus 15, 1 \oplus 6} = \bnum{10, 20, 0} = 1210.
\]
 
\end{example}

\comment{cCord}
\label{sec:orgheadline11}
We now give a Sprague-Grundy system of \(\nsigma[\ybs]\).
For an integer \(n\), let
\[
 \ord_{\ybs}(n) = \begin{cases}
 \max \Set{L \in \NN : \ybsp[L] \text{\ divides\ } n } & \tif  n \neq 0,\\
 \infty & \tif n = 0.
 \end{cases}
\]
For example, if \(\ybs = (2,3,3,2,\ldots)\), then \(\ord_{\ybs}(6) = \ord_{\ybs}(\bnum{0,0,1}) = 2\).
Consider
\begin{equation}
\label{orgtarget14}
 \sgCord[\ybs] =  \sgCord[\ybs, m] = \Set{\ccC \in \Nb^m \setminus \set{(0, \ldots, 0)}: \ord_{\ybs}\left(\sum_{i \in \Omega} \ccc^i\right) = \mord_{\ybs}(C)},
\end{equation}
where \(\mord_{\ybs}(C) = \min \set{\ord_{\ybs}(\ccc^i) :  i \in \Omega}\).
For example, if \(\ybs = (3,3,\ldots)\), then
\[
 (\bnum{1}, \bnum{0}),\ (\bnum{1,2}, \bnum{1,1}) \in \sgCord[\ybs, 2]
 \quad \tand  
 (\bnum{2}, \bnum{1}, \bnum{0})^{},\ (\bnum{1,1}, \bnum{1,1}, \bnum{1})^{} \not \in \sgCord[\ybs, 3].
\]

 \begin{theorem}
 \comment{Thm. Exist}
\label{sec:orgheadline12}
\label{orgtarget2}
The set \(\sgCord[\ybs]\) is a Sprague-Grundy system of \(\nsigma[\ybs]\).
 
\end{theorem}

\begin{remark}
 \comment{Rem.fn:flanigan}
\label{sec:orgheadline13}
\label{orgtarget15}
Flanigan \cite{flanigan-Nim-1980} showed the existence of Sprague-Grundy systems of \(\nsigma[\ybs]\) with \(\ybs = (\ybs[0],\ybs[0],\ldots)\), 
i.e., \(\ybs[L] = \ybs[0]\) for every \(L \in \NN\).
Fraenkel and Lorberbom \cite{fraenkel-Nimhoff-1991} showed that with \(\ybs = (\ybs[0],2,2,\ldots)\), i.e., \(\ybs[L] = 2\) for every \(L \ge 1\).
 
\end{remark}

\subsection{Minimal systems}
\label{sec:orgheadline19}
\label{orgtarget16}
We present a characterization of minimal systems of \(\nsigma[\ybs]\), and determine \(\ybs\) such that \(\nsigma[\ybs]\) has a minimum system.
Furthermore, we investigate minimal \emph{symmetric} systems.

We first introduce restrictions of positions.
Let \(S = \set{s^0, \ldots, s^{\cck - 1}} \subseteq \Omega\), where \(s^0 < \cdots < s^{\cck - 1}\).
For \(C \in \Nb^m\), let \(C|_S\) denote \((c^{s^0}, \ldots, c^{s^{\cck - 1}}) \in \Nb^{\cck}\),
where \(C|_{\emptyset} = () \in \Nb^0\). \footnote{When \(m = 0\), the theorems in this paper are trivial, since \(\Nb^0 = \set{()}\). For example, Theorem \ref{orgtarget2} asserts that \(\emptyset\)
 is a Sprague-Grundy system of \(\nsigma[2,0]:\set{()} \to \NN\).}
For \(\sgC \subseteq \Nb^m\), define
\[
 \sgC|_S = \big\{ \ccC|_S : \ccC \in \sgC,\ \ccC|_{\Omega \setminus S} = (\underbrace{0, \ldots, 0}_{\size{\Omega \setminus S} \text{\ times}}) \big\}.
\]
For example, if \(\sgC = \set{(1,0,0), (2,0,1), (0,2,3)}\), then
\(\sgC|_{\set{0,2}} = \set{(1,0), (2,1)}\).

We next define weights of \(\sgC\) and \(\nsigma[\ybs]\).
The \emph{weight} \(\wt(\sgC)\) of \(\sgC\) is the maximum Hamming weight of elements in \(\sgC\),
that is,
\[
 \wt(\sgC) = \max \set{\wt(\ccC) : \ccC \in \sgC},
\]
where \(\max \emptyset = 0\). 
The \emph{weight} \(\wt(\nsigma[\ybs, m])\) of the function \(\nsigma[\ybs, m]\) is the minimum weight of Sprague-Grundy
systems of \(\nsigma[\ybs, m]\), that is,
\[
 \wt (\nsigma[\ybs, m]) = \min_{\sgC \in \Delta(\nsigma[\ybs, m])} \wt(\sgC).
\]
For example, \(\wt(\nsigma[2,m])  = \wt(\sgCord[2,m][1]) = \min \set{m, 1}\).

 \begin{theorem}
 \comment{Thm. min}
\label{sec:orgheadline15}
\phantomsection
\label{orgtarget3}
\begin{enumerate}
\item \(\wt (\nsigma[\ybs, m]) = \min \set{m,\ \sup \set{\ybs[L] - 1 : L \in \NN}}\), where \(\sup P\) is the supremum of \(P\) in \(\NN \cup \set{\infty}\).
\item Let \(\ccw = \wt (\nsigma[\ybs, m])\) and \(\sgC \subseteq \Nb^m\). The following three conditions are equivalent:
\begin{enumerate}
\item \(\sgC\) is a minimal system of \(\nsigma[\ybs, m]\).
\item \(\sgC |_{S}\)  is a minimal system of \(\nsigma[\ybs, \size{S}]\) for each \(S \subseteq \Omega\).
\item \(\wt (\sgC) = \ccw\) and \(\sgC |_{S}\) is a minimal system of \(\nsigma[\ybs, \ccw]\) for each \(S \in {\Omega \choose \ccw} = \set{T \subseteq \Omega : \size{T} = \ccw}\).
\end{enumerate}
\item If \(m \ge 2\), then the function \(\nsigma[\ybs, m]\) has a minimum system if and only if \(\ybs = (\ybs[0], 2, 2, \ldots)\).
\end{enumerate}
 
\end{theorem}

\begin{remark}
 \comment{Rem. ryuo}
\label{sec:orgheadline16}
\label{orgtarget17}
Let \(\ybs = (\ybs[0], 2, 2, \ldots)\).
In \cite{fraenkel-Nimhoff-1991}, it is proved that \(\nsigma[\ybs, m]\) has a minimum system \(\sgC\).
The game \(\Gamma(\Nb^m, \sgC)\) is called \emph{cyclic Nimhoff}.
 
\end{remark}

\comment{Symmetric}
\label{sec:orgheadline17}
We next consider minimal \emph{symmetric} systems.
Let \(\sgC\) be a Sprague-Grundy system of \(\nsigma[\ybs]\) and \(C \in \sgC\).
In general, an element obtained by permuting the coordinates of \(C\)
is not necessarily in \(\sgC\). \footnote{We will give a not symmetric Sprague-Grundy system in Example \ref{orgtarget18}.}
This leads to the following definitions.
A subset \(\sgB\) of \(\Nb^m\) is said to be \emph{symmetric} if
\[
  \fS_m(B) \subseteq \sgB \quad \tforevery B \in \sgB,
\]
where \(\fS_m\) is the symmetric group of \(\set{0,1,\ldots, m - 1}\)
and
\[
 \fS_m(B) = \set{(b^{\pi(0)}, \cdots, b^{\pi(m - 1)}) : \pi \in \fS_m}.
\]
For example, \(\fS_2((0,1)) = \set{(0,1), (1,0)}\).
A symmetric Sprague-Grundy system \(\sgC\) of \(\nsigma[\ybs]\) is called a \emph{minimal symmetric system}
if \(\sgC\) is minimal among all symmetric Sprague-Grundy systems of \(\nsigma[\ybs]\).
If \(\nsigma[\ybs]\) has a unique minimal symmetric system, then it is called the \emph{minimum symmetric system} of \(\nsigma[\ybs]\).

 \begin{theorem}
 \comment{Thm.}
\label{sec:orgheadline18}
\label{orgtarget19}
If \(m \ge 2\), then the function \(\nsigma[\ybs, m]\) has a minimum symmetric system
if and only if \(\ybs = (\ybs[0], 2, 2, \ldots)\) or \(\ybs = (2, 3, 2, 2, \ldots)\).
 
\end{theorem}

\subsection{Maximum systems}
\label{sec:orgheadline28}
\label{orgtarget20}
We will give a recursive construction of the maximum system of \(\nsigma[\ybs]\).

Let \(\sgCmax[\ybs] = \sgCmax[\ybs,m] = \max \Delta(\nsigma[\ybs, m])\).
By the definition of Sprague-Grundy functions,
\begin{equation}
\label{orgtarget21}
 \sgCmax[\ybs] = \set{C \in \Nb^m: \nsigma[\ybs](X + C) \neq \nsigma[\ybs](X) \tforevery X \in \Nb^m}
\end{equation}
(see Section \ref{orgtarget10}).
We begin by comparing \(\sgCmax[\ybs]\) and \(\sgCord[\ybs]\) defined in (\ref{orgtarget14}).
In fact, for \(m \ge 2\), we can show that  \(\sgCmax[\ybs, m] = \sgCord[\ybs, m]\) if and only if \(\ybs = (\ybs[0], 2, 2, \ldots)\)
(Theorem \ref{orgtarget22}).

 \begin{example}
 \comment{Exm. CandFbinary}
\label{sec:orgheadline20}
\label{orgtarget23}
Let \(\ybs = (2,2,\ldots)\). Then the maximum system  \(\sgCmax[\ybs]\) is simple.
Consider \(\ccC = (1,2) = (\bnum{1}, \bnum{0,1})\) and \(\ccF = (1, 3) = (\bnum{1}, \bnum{1,1})\).
Then \(C \in \sgCord[\ybs]\) and \(\ccF \not \in \sgCord[\ybs]\), since \(\ord_{\ybs} (\sum c^i) = 0 = \mord_{\ybs}(C)\) and \(\ord_{\ybs} (\sum f^i) = 2 > 0 = \mord_{\ybs}(F)\).
Moreover, \(C \in \sgCmax[\ybs]\) and \(\ccF \not \in \sgCmax[\ybs]\) because
\[ 
\nsigma[\ybs]<0>(X + \ccC) = \nsigma[\ybs]<0>(X) \oplus 1 \neq \nsigma[\ybs]<0>(X) \tforevery X \in \Nb^2
\]
and \(\nsigma[\ybs]((1, 0) + \ccF) = \nsigma[\ybs](2,3) = 1 = \nsigma[\ybs]((1, 0))\).
As we have mentioned, the maximum system \(\sgCmax[\ybs]\) is actually equal to \(\sgCord[\ybs]\) in this case.
 
\end{example}

 \begin{example}
 \comment{Exm. CandF}
\label{sec:orgheadline21}
\label{orgtarget24}
Let \(\ybs = (3,3,\ldots)\).
Then the maximum system \(\sgCmax[\ybs]\) is not as simple as \(\sgCmax[(2,2,\ldots)]\).
Consider \(\ccC = (2, 10) = (\bnum{2}, \bnum{1,0,1})\) and \(\ccF = (2, 4) = (\bnum{2}, \bnum{1,1})\).
Then \(\ccC \not \in \sgCord[\ybs]\) and \(\ccF \not \in \sgCord[\ybs]\), since \(\ord_{\ybs} (\sum c^i) = 1 > 0 = \mord_{\ybs}(C)\)
and \(\ord_{\ybs} (\sum f^i) = 1 > 0 = \mord_{\ybs}(F)\).
However, \(\ccC \in \sgCmax[\ybs]\) and \(\ccF \not \in \sgCmax[\ybs]\).
Indeed, \(\nsigma[\ybs]((1, 2) + \ccF) = \nsigma[\ybs](3, 6) = 0 = \nsigma[\ybs]((1, 2))\),
so \(\ccF \not \in \sgCmax[\ybs]\).
A direct computation shows that
\[
 \nsigma[\ybs]<1>(X + C) \neq \nsigma[\ybs]<1>(X) \quad \tor \quad \nsigma[\ybs]<2>(X + C)\neq \nsigma[\ybs]<2>(X) \quad \tforevery X \in \Nb^2.
\]
Therefore \(\ccC \in \sgCmax[\ybs] \setminus \sgCord[\ybs]\).
In particular, \(\sgCmax[\ybs] \supsetneq \sgCord[\ybs]\).
We will present a systematic way to determine whether \(\ccC \in \sgCmax[\ybs]\) at the end of this subsection.
 
\end{example}

\comment{Algorithm}
\label{sec:orgheadline22}
We now give a recursive construction of \(\sgCmax[\ybs]\).
Set \(\cF^{\ybs} = \cF^{\ybs, m} = \Nb^m \setminus \sgCmax[\ybs, m]\).
Let \(\Chopped{\ybs} = \ybs[\ge 1] = (\beta_1, \beta_2, \ldots)\).
For \(n \in \Nb\) and \(F \in \Nb^m\), let
\(\Chopped{n} = n_{\ge 1} = [n_1, n_2 \ldots] \in \Nb[\Chopped{\ybs}]\) and \(\Chopped{F} = (\Chopped{f^0}, \ldots, \Chopped{f^{m - 1}}) \in \Nb[\Chopped{\ybs}]^{m}\).
Define
\begin{equation}
\label{orgtarget25}
 \Carry(F) = \set{\ccr \in \set{0,1}^m : r^i \le \ccf^i_0 \tforevery i \in \Omega}.
\end{equation}
For example, if \(\ybs = (4,4,\ldots)\) and \(F = (\bnum{3}, \bnum{0,2}) \in \Nb^2\), then
\(\Carry(F) = \set{0,1} \times \set{0} = \set{(0,0), (1,0)}\).
We will see that \(\Carry(F)\) is derived from carry in Subsection \ref{orgtarget26}.
For \(L \in \NN\), define
\[
 \sgFA<L> = \sgFA[\ybs, m]<L> = \Set{F \in \Nb^m : \nsigma[\ybs]<0>(F) = 0, \quad \Chopped{F} + r \in \sgFA[\Chopped{\ybs}, m]<L - 1> \tforsome r \in \Carry(F)},
\]
where \(\sgFA[\Chopped{\ybs}, m]<-1> = \set{(0,\ldots, 0)}\).
For example, if \(\ybs = (3,3,\ldots)\) and \(F = (\bnum{2}, \bnum{1,1}) \in \Nb^2\), then 
\[
 \sgFA[\ybs, 2]<0> = \set{(0,0), (1,2), (2,1)}
\]
and \(F \in \sgFA[\ybs, 2]<1>\). Indeed, \(\Carry(F) = \set{0,1}^2 \ni (1,1)\) and
\[
 \Chopped{F} + (1,1) = (0,1) + (1,1) = (1,2) \in \sgFA[\ybs, 2]<0> = \sgFA[\Chopped{\ybs}, 2]<0>.
\]
Since \(\nsigma[\ybs]<0>(F) = 2 \oplus 1 = 0\), we see that \(F \in \sgFA[\ybs, 2]<1>\).

 \begin{theorem}
 \comment{Thm.}
\label{sec:orgheadline23}
\label{orgtarget4}
For \(L \in \NN\),
\begin{equation}
\label{orgtarget27}
 \Set{F \in \sgF^{\ybs} : \max F \le \ybsp[L + 1] - \ybsp[L]} \subseteq \sgFA<L>  \subseteq \sgF^{\ybs}.
\end{equation}
In particular,
\begin{enumerate}
\item \(\displaystyle \sgF^{\ybs} = \bigcup_{L \in \NN} \sgFA<L>\).
\item \(F \in \sgFA\) if and only if \(\nsigma[\ybs]<0>(F) = 0\) and \(\Chopped{F} + r \in \sgF^{\Chopped{\ybs}}\) for some \(r \in \Carry(F)\).
\item \(\displaystyle \sgCmax[\ybs] = \Set{C \in \Nb^m : \max C \le \ybsp[L + 1] - \ybsp[L] \tand\  \ccC \not \in  \sgFA<L>  \tforsome L \in \NN}\).
\end{enumerate}
 
\end{theorem}

\comment{How to check}
\label{sec:orgheadline24}
The assertion (2) of Theorem \ref{orgtarget4} allows us to determine whether \(F \in \sgFA\) or \(F \in \sgCmax[\ybs]\).

 \begin{example}
 \comment{Exm. CandF2}
\label{sec:orgheadline25}
\label{orgtarget28}
Let \(C\) be as in Example \ref{orgtarget24}.
To verify that \(C \in \sgCmax[\ybs]\),
it suffices to show that \(\Chopped{C} + r \not \in \sgF^{\Chopped{\ybs}}\) for every \(r \in \Carry(C)\).
Note that if \(\Chopped{C} + r  \in \sgF^{\Chopped{\ybs}}\), then \(\nsigma[\Chopped{\ybs}]<0>(\Chopped{C} + r) = 0\).
Since \(\Carry(C) = \set{0,1}^2\),
it follows that if \(\ccr \in \Carry(C)\) and \(\nsigma[\ybs]<0>(\Chopped{C} + \ccr) = 0\), then \(\ccr = (0,0)\).
Moreover, since \(\Chopped{C} = (\bnum{0}, \bnum{0,1})\) and \(\rho(\Chopped{C}) = \set{0}^2\),
we see that \(\Chopped{C} \not \in \sgF^{\Chopped{\ybs}}\).
Consequently \(\ccC \not \in \sgF^{\ybs}\), that is, \(\ccC \in \sgCmax[\ybs]\).
 
\end{example}

 \begin{theorem}
 \comment{Thm.}
\label{sec:orgheadline26}
\label{orgtarget22}
If \(m \ge 2\), then \(\sgCmax[\ybs, m] = \sgCord[\ybs, m]\) if and only if \(\ybs = (\ybs[0], 2, 2, \ldots)\).
 
\end{theorem}

\begin{remark}
 \comment{Rem.}
\label{sec:orgheadline27}
Blass, Fraenkel, and Guelman \cite{blass-how-1998} proved that if
\(\ybs = (2, 2, 2, \ldots)\), then \(\sgCmax[\ybs, m] = \sgCord[\ybs, m]\).
 
\end{remark}

\subsection{Organization}
\label{sec:orgheadline29}
This paper is organized as follows.
In Section \ref{orgtarget10}, we recall the concept of impartial games.
In Section \ref{orgtarget29}, we prove Theorem \ref{orgtarget2} that states that \(\sgCord[\ybs]\) is a Sprague-Grundy system of \(\nsigma[\ybs]\).
In Sections \ref{orgtarget30} and \ref{orgtarget31}, we investigate minimal systems and maximum systems, respectively.
The proofs of Theorems \ref{orgtarget3} and \ref{orgtarget19} are in Section \ref{orgtarget30},
and that of Theorems \ref{orgtarget4} and \ref{orgtarget22} are in Section \ref{orgtarget31}.

\section{Games}
\label{sec:orgheadline33}
\label{orgtarget10}
We recall the concept of impartial games. See \cite{berlekamp-Winning-2001,conway-numbers-2001} for details.

\comment{Games}
\label{sec:orgheadline31}
Let \(\sgC \subseteq \Nb^m \setminus \set{(0,\ldots,0)}\) and \(\Gamma = \Gamma(\Nb^m, \sgC)\).
An element of \(\Nb^m\) is called a \emph{position} of the game \(\Gamma\).
Let \(X\) and \(Y\) be two positions in \(\Gamma\).
If \(X - Y \in \sgC\), then \(Y\) is called an \emph{option} of \(X\) (in \(\Gamma\)).
If \(X - Y \in \Nb^m\), then \(Y\) is called a \emph{descendant} of \(X\).
The \emph{Sprague-Grundy value} \(\sg_{\Gamma}(X)\) of \(X\) is defined to be the minimum nonnegative integer \(n\)
such that \(X\) has no option \(Y\) with \(\sg_{\Gamma}(Y) = n\). The function \(\sg_{\Gamma} : \Nb^m \to \Nb\) is called the \emph{Sprague-Grundy function} of the game \(\Gamma\).

\comment{Cover}
\label{sec:orgheadline32}
In this paper, for \(h \in \Nb\),
we say that \(\sgC\) \emph{covers} \(X\) at \(h\) (with respect to \(\nsigma[\ybs]\))
if \(X - C \in \Nb^m\) and \(\nsigma[\ybs](X - C) = h\) for some \(C \in \sgC\), that is,
\(X\) has an option \(Y\) with \(\nsigma[\ybs](Y) = h\) in \(\Gamma\).
When \(\sgC\) is a singleton \(\set{C}\), we also say that \(C\) covers \(X\) at \(h\).
If \(\sgC\) covers \(X\) at every \(h\) with \(0 \le h < \nsigma[\ybs](X)\), then
we say that \(\sgC\) \emph{adequately covers} \(X\).
Observe that \(\sgC\) is a Sprague-Grundy system of the function \(\nsigma[\ybs]\) if and only if the following two conditions hold:
\phantomsection
\begin{description}
\item[{(SG1)}] \label{orgtarget32} \(\sgC\) does not cover \(X\) at \(\nsigma[\ybs](X)\) for every \(X \in \Nb^m\).
\end{description}
\phantomsection
\begin{description}
\item[{(SG2)}] \label{orgtarget33} \(\sgC\) adequately covers \(X\) for every \(X \in \Nb^m\).
\end{description}
It is now easy to see that (\ref{orgtarget9}) and (\ref{orgtarget21}) hold.

\section{Sprague-Grundy systems}
\label{sec:orgheadline61}
\label{orgtarget29}
We prove Theorem \ref{orgtarget2}.
In addition, we present the basic properties of Sprague-Grundy
systems of \(\nsigma[\ybs]\) to prove Theorem \ref{orgtarget3}.
\subsection{Proof of Theorem \texorpdfstring{\ref{orgtarget2}}{1.5}}
\label{sec:orgheadline38}
\label{orgtarget34}

\comment{SG1}
\label{sec:orgheadline34}
We first verify that \(\sgCord[\ybs]\) satisfies (\hyperref[orgtarget32]{SG1}).
Let \(C \in \sgCord[\ybs]\) and \(X \in \Nb^m\) with \(X - \ccC \in \Nb^m\).
Set \(Y = X - \ccC\),  \(h = \nsigma[\ybs](Y)\), and \(n = \nsigma[\ybs](X)\).
Let \(N = \mord_{\ybs}(C)\).
It suffices to show that \(h_N \neq n_N\).
By the definition of \(N\), we see that \(c^i_L = 0\) for every \(L < N\), and hence \(y^i_N = (x^i - c^i)_N =  x^i_N \ominus \ccc^i_N\). This implies that
\begin{equation}
\label{orgtarget35}
 h_N = y^0_N \oplus \cdots \oplus y^{m - 1}_N = x_N^0 \oplus \cdots \oplus x^{m - 1}_N \ominus (\ccc^0_N \oplus \cdots \oplus \ccc^{m - 1}_N) = n_N \ominus (\ccc^0_N \oplus \cdots \oplus \ccc^{m - 1}_N).
\end{equation}
Since \(C \in \sgCord[\ybs]\), it follows that \(N = \ord_{\ybs} (\sum \ccc^i)\), and hence that \((\sum \ccc^i)_N = \ccc^0_N \oplus \cdots \oplus \ccc^{m - 1}_N \neq 0\).
Combining this with (\ref{orgtarget35}), we see that \(h_N \neq n_N\).

\comment{SG2}
\label{sec:orgheadline35}
It remains to prove that \(\sgCord[\ybs]\) satisfies (\hyperref[orgtarget33]{SG2}).
Let us consider an example.
For \(X \in \Nb^m\), we write
\[
 X = \begin{bmatrix}
 x^0_0 & x^0_1 & \cdots & x^0_N  \\
 \vdots & \vdots & \ddots & \vdots \\
 x^{m - 1}_0 & x^{m - 1}_1 & \cdots & x^{m - 1}_N \\
 \end{bmatrix}
\]
if \(x^i_L = 0\) for every \(L > N\) and every \(i \in \Omega\).
 \begin{example}
 \comment{Exm. How to choose options}
\label{sec:orgheadline36}
\label{orgtarget36}
Let \(\ybs = (4,4,\ldots)\), and consider
\[
X =\begin{bmatrix}
 1 & 2 & 2 & 1 & 0 \\
 2 & 1 & 1 & 1 & 0 \\
 0 & 1 & 2 & 1 & 1 \\
\end{bmatrix}\in \Nb^3.
\]Set \(n = \nsigma[\ybs](X) = \bnum{3, 0, 1, 3, 1} \in \Nb\) and \(h = \bnum{0,1,2, 0, 1} \in \Nb\).
We verify that \(\sgCord[\ybs]\) covers \(X\) at \(h\).
Note that \(\max \set{L \in \NN : n_L \neq h_L} = 3\).
We can choose a descendant \(Y\) of \(X\) so that \(\nsigma[\ybs](Y) = h\) and
\[
Y =\begin{bmatrix}
 y^0_0 & y^0_1 & y^0_2 & y_3^0 & x^0_4 \\
 x^1_0 & x^1_1 & x^1_2 & y_3^1 & x^1_4 \\
 x^2_0 & x^2_1 & x^2_2 & y_3^2 & x^2_4 \\
\end{bmatrix}=
\begin{bmatrix}
 y^0_0 & y^0_1 & y^0_2 & y_3^0 & 0 \\
 2 & 1 & 1 & y_3^1 & 0 \\
 0 & 1 & 2 & y_3^2 & 1 \\
\end{bmatrix}.
\]Indeed, let \(y^0 = \bnum{2, 3, 3, 0, 0}\) and \(y^1_3 = y^2_3 = 0\). Then \(\nsigma[\ybs](Y) = h\).
Since
\[
X - Y =\begin{bmatrix}
 3 & 2 & 2 & 0 \\
 0 & 0 & 0 & 1 \\
 0 & 0 & 0 & 1 \\
\end{bmatrix}\in \sgCord[\ybs],
\]it follows that \(\sgCord[\ybs]\) covers \(X\) at \(h\).
 
\end{example}

\comment{stronger}
\label{sec:orgheadline37}
It is now easy to verify that \(\sgCord[\ybs]\) adequately covers every position in \(\Nb^m\).
Instead of proving this directly, we present a slightly stronger result in the next subsection.

\subsection{Smaller Sprague-Grundy systems}
\label{sec:orgheadline49}
\label{orgtarget37}
Example \ref{orgtarget36} suggests that we may obtain a smaller Sprague-Grundy system, which will be used to study minimal systems, than \(\sgCord[\ybs]\).
Let \(\delta(n) = 1\) if \(n = 0\), and \(\delta(n) = 0\) if \(n \neq 0\).
For \(C \in \Nb^m\), let \(Nj(C)\) denote the set of \((N, j) \in \NN^2\) satisfying the following three conditions:
\phantomsection
\begin{description}
\item[{(C1)}] \label{orgtarget38} \(\displaystyle c^i = \begin{cases} 
 c^j_{\le N}  = \bnum{c^j_0, \ldots, c^j_N} \neq 0 & \tif i = j, \\ 
 c^i_N \ybsp[N]  = \bnum{0, \ldots, 0, c^i_N} & \tif i \neq j. \end{cases}\)
\end{description}
\phantomsection
\begin{description}
\item[{(C2)}] \label{orgtarget39} \(\sum_{i} c^i_N + \delta(c^j_N)  \le \ybs[N] - 1\).
\end{description}
\phantomsection
\begin{description}
\item[{(C3)}] \label{orgtarget40} \(c^i_N \le c^{j}_N + 1\) for every \(i \in \Omega\).
\end{description}
For example, if \(X\) and \(Y\) are as in Example \ref{orgtarget36}, then \(Nj(X - Y) = \set{(3, 0)}\).
Define 
\[
 \sgCnmin[\ybs] = \sgCnmin[\ybs, m] = \Set{C \in \Nb^m : Nj(C) \neq \emptyset}.
\]
We will prove that \(\sgCnmin[\ybs] \subseteq \sgCord[\ybs]\) in Lemma \ref{orgtarget41} and that \(\sgCnmin[\ybs] \in \Delta(\nsigma^{\ybs})\) in Lemma \ref{orgtarget42},
which means that \(\sgCord[\ybs] \in \Delta(\nsigma^{\ybs})\).

 \begin{example}
 \comment{Exm. cCnmin}
\label{sec:orgheadline39}
\label{orgtarget43}
Let \(C \in \sgCnmin[\ybs]\), and suppose that \((N, 0) \in Nj(C)\).
Then \(c^0 \neq 0\) and
\begin{equation}
\label{orgtarget44}
C =\begin{bmatrix}
 c^0_0 & \cdots & c^0_{N - 1} & c^0_{N} \\
 0 & \cdots & 0 & c^1_{N} \\
 \vdots & \ddots & \vdots & \vdots \\
 0 & \cdots & 0 & c^{m - 1}_N \\
\end{bmatrix}.
\end{equation}
For example, let \(\ybs = (5,5,\ldots)\) and consider
\[
C^{(0)} =\begin{bmatrix}
 1 & 0 \\
 0 & 1 \\
 0 & 1 \\
\end{bmatrix}\hspace{-0.3em},\
C^{(1)} =\begin{bmatrix}
 0 & 1 \\
 0 & 0 \\
 0 & 0 \\
\end{bmatrix}\hspace{-0.3em},\
C^{(2)} =\begin{bmatrix}
 0 & 2 \\
 0 & 1 \\
 0 & 1 \\
\end{bmatrix}\hspace{-0.3em},\
\]\\
\[
D^{(0)} =\begin{bmatrix}
 1 & 1 \\
 1 & 1 \\
 0 & 1 \\
\end{bmatrix}\hspace{-0.3em},\
D^{(1)} =\begin{bmatrix}
 1 & 3 \\
 0 & 2 \\
 0 & 0 \\
\end{bmatrix}\hspace{-0.3em}.
\]By definition,
\[ 
 Nj(C^{(0)}) = \Set{(1,0)},\quad  Nj(C^{(1)}) = \Set{(L ,0) : L \ge 1},
\]
and 
\[ 
Nj(C^{(2)}) = \Set{(1, j) :0 \le j \le 2}.
\]
Hence \(C^{(0)}, C^{(1)}, C^{(2)} \in \sgCnmin[\ybs]\). It is also easy to check that \(D^{(0)}, D^{(1)} \not \in \sgCnmin[\ybs]\).
 
\end{example}

\begin{remark}
 \comment{Rem. Nandj}
\label{sec:orgheadline40}
\label{orgtarget45}
For \(C \in \Nb^m\), define
\[
 N(C) = \set{N \in \NN : (N, j) \in Nj(C) \tforsome j \in \NN}
\]
and 
\[
 j(C) = \set{j \in \NN : (N, j) \in Nj(C) \tforsome N \in \NN}.
\]
Then \(Nj(C) = N(C) \times j(C)\).
More precisely,
\[
 Nj(C) = \begin{cases}
  \set{L \in \NN : L \ge N} \times \set{j} \tforsome N, j \in \NN & \tif \wt(C) = 1, \\
  \set{N} \times J \tforsome N \in \NN  \text{\ \ and some\ \ } J \subseteq \Omega & \tif \wt(C) \neq 1.
 \end{cases}
\]
Indeed, suppose that \(\wt(C) = 1\), and choose \(j \in \Omega\) such that \(c^j \neq 0\).
Set \(N = \max \{L \in \NN : c^j_L \neq 0 \}\).
Then \(Nj(C) = \set{L \in \NN : L \ge N} \times \set{j}\).
Next, suppose that \(\wt(C) \neq 1\).
If \(Nj(C) = \emptyset\), then \(Nj(C) = \emptyset = \set{0} \times \emptyset\).
Suppose that \(Nj(C) \neq \emptyset\), and let \((N, j) \in Nj(C)\). Then \(c^j \neq 0\), so \(\wt(C) \ge 2\) since \(\wt(C) \neq 1\).
Hence \(c^i_N \neq 0\) for some \(i \neq j\) by (\hyperref[orgtarget38]{C1}).
By (\hyperref[orgtarget38]{C1}) again, we see that \(N(C) = \set{N}\) and \(Nj(C) = \set{N} \times J\) for some \(J \subseteq \Omega\).
 
\end{remark}

 \begin{lemma}
 \comment{Lem. cCnmin}
\label{sec:orgheadline41}
\label{orgtarget41}
\(\sgCnmin[\ybs, m] \subseteq \sgCord[\ybs,m]\).
 
\end{lemma}

\begin{proof}
 \comment{Proof.}
\label{sec:orgheadline42}
For \(C \in \sgCnmin[\ybs, m]\),
we show that
\[
 \ord_{\ybs} \left(\sum c^i \right) = \mord_{\ybs}(C).
\]
Let \(M = \mord_{\ybs}(C)\).
Since \(\ord_{\ybs} \left(\sum c^i \right) \ge M\) and \((\sum c^i)_M = c^0_M \oplus \cdots \oplus c^{m - 1}_M\),
it suffices to show that \(c^0_M \oplus \cdots \oplus c^{m -1}_M \neq 0\).
Let \((N, j) \in Nj(C)\). Then \(N \ge M\) since, by (\hyperref[orgtarget38]{C1}), \(c^i_L = 0\) for every \(L > N\) and every \(i \in \Omega\).
Suppose that \(N > M\). By (\hyperref[orgtarget38]{C1}), if \(i \neq j\), then \(c^i_M\) = 0, and hence  \(c^j_M \neq 0\).
This implies that \(c^0_M \oplus \cdots \oplus c^{m -1}_M = c^j_M \neq 0\).
Suppose that \(N = M\). Then \(\sum c^i_M \le \ybs[M] - 1\) by (\hyperref[orgtarget39]{C2}).
Therefore \(c^0_M \oplus \cdots \oplus c^{m - 1}_M = \sum c^i_M \neq 0\).

\end{proof}

\comment{Connect}
\label{sec:orgheadline43}
We now show that \(\sgCnmin[\ybs]\) is a Sprague-Grundy system of \(\nsigma[\ybs]\).
We present a slightly stronger result than this to investigate minimal systems in Section \ref{orgtarget30}.
 \begin{lemma}
 \comment{Lem. exist}
\label{sec:orgheadline44}
\label{orgtarget42}
Let \(X \in \Nb^m, n = \nsigma[\ybs](X)\), and \(h \in \Nb\) with \(h < n\). Set \(N = \max \{L \in \NN : \pexp{\ccn}[L] \neq \pexp{\cch}[L]\}\). 
Choose \(j \in \Omega\) so that \(x^j_{\le N} \ge x^i_{\le N}\) for every \(i \in \Omega\).
Then there exists \(\ccC \in \sgCnmin[\ybs]\) satisfying the following four conditions:
\phantomsection
\begin{description}
\item[{(A1)}] \label{orgtarget46} \(\nsigma[\ybs](X - \ccC) = h\).
\end{description}
\phantomsection
\begin{description}
\item[{(A2)}] \label{orgtarget47} \((N, j) \in Nj(\ccC)\).
\end{description}
\phantomsection
\begin{description}
\item[{(A3)}] \label{orgtarget48} \(X_{\le N} - C \in \Nb^m\), where \(X_{\le N} = (x^0_{\le N}, \ldots, x^{m - 1}_{\le N})\).
\end{description}
\phantomsection
\begin{description}
\item[{(A4)}] \label{orgtarget49} \((x^{j} - \ccc^{j})_N < x^{j}_N\).
\end{description}
In particular, \(\sgCnmin[\ybs]\) is a Sprague-Grundy system of \(\nsigma[\ybs]\).
 
\end{lemma}

\begin{proof}
 \comment{Proof.}
\label{sec:orgheadline45}
We first construct a descendant \(Y\) of \(X\) such that \(\nsigma[\ybs](Y) = h\) in the same way as in Example \ref{orgtarget36}.
By rearranging \(x^i\), we may assume that \(j = 0\).
Since \(\cch < \ccn\), we see that
\(\pexp{\cch}[N] < \pexp{\ccn}[N] = \pexp{x^0}[N] \oplus \cdots \oplus \pexp{x^{\ccm - 1}}[N]\).
Hence we can choose \((y^0_N, \ldots, y^{m - 1}_N) \in \range{\ybs[N]}^m\) satisfying the following two conditions:
\begin{enumerate}
\item \(0 \le x^i_N - y^i_N \le x^0_N - y^0_N \;\; \tforevery i \in \Omega\).
\item \(\sum_{i} (x^i_N - y^i_N) = n_N \ominus h_N\). \footnote{For example, if \(\ybs[N] = 6, m = 3,  (x^0_N, x^1_N, x^2_N) = (2,2,1)\), and \(h_N = 2\), then \((0,1,1)\) satisfies (1) and (2).}
\end{enumerate}
Then \(y^0_N < x^0_N\) and
\[
 y^0_N \oplus \cdots \oplus y^{m - 1}_N = x^0_N \oplus \cdots \oplus x^{m - 1}_N \ominus n_N \oplus h_N = h_N.
\]
Consider
\[
Y =\begin{bmatrix}
 x^0_0  \ominus \ccn_0  \oplus  \cch_0 & \cdots & x^0_{N - 1}  \ominus n_{N - 1}  \oplus  h_{N - 1} & y^0_N & x^0_{N + 1} & x^0_{N + 2} & \cdots \\
 x^1_0 & \cdots & x^1_{N - 1} & y^1_N & x^1_{N + 1} & x^1_{N + 2} & \cdots \\
 \vdots & \ddots & \vdots & \vdots & \vdots & \vdots & \ddots \\
 x^{m - 1}_0 & \cdots & x^{m - 1}_{N - 1} & y^{m - 1}_N & x^{m - 1}_{N + 1} & x^{m - 1}_{N + 2} & \cdots \\
\end{bmatrix}\in \Nb^m.
\]Then \(Y\) is a descendant of \(X\) with \(\nsigma[\ybs](Y) = \cch\).

Let \(C = X - Y\). 
We next verify that \((N, 0) \in Nj(C)\).

(\hyperref[orgtarget38]{C1}). Since \(c^0 = x^0 - y^0 = x^0_{\le N} - y^0_{\le N} \neq 0\) and \(c^i = x^i - y^i = (x^i_N - y^i_N) \ybsp[N]\) for every \(i > 0\), (\hyperref[orgtarget38]{C1}) holds.

(\hyperref[orgtarget39]{C2}). For \(i > 0\), we know that \(c_N^i = (x^i - y^i)_N = x_N^i - y^i_N\).
Write \(c_N^0 = x_N^0 - y_N^0 - \epsilon^0_N\).
Then \(\epsilon^0_N \in \set{0,1}\) and \(\epsilon^0_{N} = 1\) when there is a borrow in the \(N\)th digit in the calculation of \(x^0 - y^0\) in base \(\ybs\).
By (2),
\[
 \sum_{i = 0}^{m - 1} c_N^i = \sum_{i = 0}^{m - 1} (x^i_N - y^i_N) - \epsilon^0_N = n_N \ominus h_N - \epsilon^0_N \le \ybs[N] - 1 - \epsilon^0_N.
\]
Thus (\hyperref[orgtarget39]{C2}) holds when \(c_N^0 \neq 0\).
Suppose that \(c_N^0 = 0\).
Then \(\epsilon^0_N = 1\) since \(c_N^0 = x^0_N - y^0_N - \epsilon^0_N\) and \(x^0_N - y^0_N \ge 1\).
Hence (\hyperref[orgtarget39]{C2}) holds.

(\hyperref[orgtarget40]{C3}). This is trivial when \(i = 0\).
For \(i > 0\), by (1), \(c^i_N  = x^i_N - y^i_N \le x^0_N - y^0_N =  c^0_N + \epsilon^0_N \le c^0_N + 1\), and hence (\hyperref[orgtarget40]{C3}) holds.

Therefore \((N, 0) \in Nj(C)\) and  \(C \in \sgCnmin[\ybs]\).

Finally, we show (\hyperref[orgtarget46]{A1})-(\hyperref[orgtarget49]{A4}).
We have already shown (\hyperref[orgtarget46]{A1}) and (\hyperref[orgtarget47]{A2}).
Since \(c^i = x^i - y^i = x^i_{\le N} - y^i_{\le N} \le x^i_{\le N}\), (\hyperref[orgtarget48]{A3}) holds.
Since \((x^0 - c^0)_N = y^0_N  < x^0_N\), (\hyperref[orgtarget49]{A4}) also holds.

\end{proof}

\comment{connect}
\label{sec:orgheadline46}
For a general Sprague-Grundy system \(\sgC \in \Delta(\nsigma[\ybs])\), we can obtain a weaker result than Lemma \ref{orgtarget42}.
 \begin{lemma}
 \comment{Lem. under}
\label{sec:orgheadline47}
\label{orgtarget50}
Let \(X, n, h\), and \(N\) be as in Lemma \ref{orgtarget42}.
If \(\sgC \in \Delta(\nsigma[\ybs])\), then there exists \(C \in \sgC\) satisfying the following two conditions:
\begin{description}
\item[{(A1)}] \(\nsigma[\ybs](X - C) = h\).
\item[{(A3)}] \(X_{\le N} - C \in \Nb^m\).
\end{description}
 
\end{lemma}

\begin{proof}
 \comment{Proof.}
\label{sec:orgheadline48}
Since \(\nsigma[\ybs](X_{\le N}) = n_{\le N} > h_{\le N}\),
there exists \(C \in \sgC\) such that \(X_{\le N} - C \in \Nb^m\) and \(\nsigma[\ybs](X_{\le N} - C) = h_{\le N}\).
Since \(h_{L} = n_{L}\) for every \(L \ge N + 1\), it follows that \(\nsigma[\ybs](X - C) = h\).

\end{proof}

\subsection{Properties of Sprague-Grundy systems}
\label{sec:orgheadline60}
\label{orgtarget51}
To prove Theorem \ref{orgtarget3}, we will present the basic properties of Sprague-Grundy systems of \(\nsigma[\ybs]\).

If \(\sgC \in \Delta(\nsigma[\ybs])\) and \(X \in \Nb^m\), then \(\sgC\) covers \(X\) at every \(h\) with \(0 \le h < \nsigma[\ybs](X)\).
The next lemma ensures that \(\sgC\) covers \(X\) also at some \(h\) with \(h > \nsigma[\ybs](X)\).
 \begin{lemma}
 \comment{Lem. over}
\label{sec:orgheadline50}
\label{orgtarget52}
Let \(\sgC\) be a subset of \(\Nb^m\) satisfying {\normalfont{(\hyperref[orgtarget33]{SG2})}}.
Let \(X \in \Nb^m, n = \nsigma[\ybs](X), h \in \Nb\) with \(h \neq n\), and \(N = \max \{L \in \NN : \ccn_L \neq \cch_L\}\).
If \(\sum_{i} x^i_N \ge \ybs[N] - 1\), then \(\sgC\) covers \(X\) at \(\cch\).
 
\end{lemma}

\begin{proof}
 \comment{Proof.}
\label{sec:orgheadline51}
If \(\sum x^i_N = \ybs[N] - 1\), then \(n_N = \sum x^i_N = \ybs[N] - 1 > h_N\), so \(\sgC\) covers \(X\) at \(h\).
Suppose that \(\sum x^i_N \ge \ybs[N]\).
Then there exist \(\tilde{x}^0_N, \ldots, \tilde{x}^{m - 1}_N \in \range{\ybs[N]}\) satisfying the following two conditions:
\begin{enumerate}
\item \(\tilde{x}^i_N \le x^i_N\).
\item \(\sum_i \tilde{x}^i_N = \ybs[N] - 1\).
\end{enumerate}
Let \(\tilde{x}^i = \bnum{x^i_0, \ldots, x^i_{N - 1}, \tilde{x}^i_N} \in \Nb, \tilde{X} = (\tilde{x}^0, \ldots, \tilde{x}^{m - 1})\), and \(\tilde{n} = \nsigma[\ybs](\tilde{X})\).
Put \(g = h \ominus n\).
Since \(N = \max \set{L \in \NN : g_L \neq 0}\) and \(\tilde{n}_N = \sum \tilde{x}^i_N = \ybs[N] - 1\),
it follows that \(\tilde{n} \oplus \ccg < \tilde{n}\).
Hence there exists \(C \in \sgC\) such that \(C\) covers \(\tilde{X}\) at \(\tilde{n} \oplus \ccg\).
We show that \(C\) covers \(X\) at \(\ccn \oplus \ccg (= \cch)\).
Since \(\tilde{X} - C \in \Nb^m\), we see that \(X - C \in \Nb^m\).
To show that  \(\nsigma[\ybs](X - C) = n \oplus g\), it suffices to verify that \(\nsigma[\ybs]<\ccL>(X - C) = \ccn_\ccL \oplus \ccg_\ccL\) for every \(\ccL \in \NN\).
If \(0 \le \ccL < N\), then \(x^i_\ccL = \tilde{x}^i_\ccL\), and hence
\[
 \nsigma[\ybs]<\ccL>(X - C) = \nsigma[\ybs]<L>(\tilde{X} - C) = \tilde{\ccn}_\ccL \oplus \ccg_\ccL = \ccn_\ccL \oplus \ccg_\ccL.
\]
Suppose that \(\ccL = N\), and write
\((x^i - c^i)_N = x^i_N \ominus c^i_N \ominus \epsilon^i_N\).
Then \((\tilde{x}^i - c^i)_N = \tilde{x}^i_N \ominus c^i_N \ominus \epsilon^i_N\), since \(\tilde{x}^i_\ccM =  x^i_\ccM\) for every \(\ccM < N\).
Because \(n_N = x^0_N \oplus \cdots \oplus x^{m - 1}_N\) and \(\tilde{n}_N = \tilde{x}^0_N \oplus \cdots \oplus \tilde{x}^{m - 1}_N\), we see that
\[
 \nsigma[\ybs]<N>(X - C) \ominus \ccn_N = \bigominus_{i} (c_N^i \oplus \epsilon_N^i) =  \nsigma[\ybs]<N>(\tilde{X} - C)\ominus \tilde{n}_N = \ccg_N,
\]
and hence \(\nsigma[\ybs]<N>(X - C) =  n_N \oplus \ccg_N\).
If \(\ccL \ge N + 1\), then \((x^i - c^i)_\ccL = x^i_\ccL\) and \(\ccg_\ccL = 0\), so \(\nsigma[\ybs]<\ccL>(X - C) = \ccn_\ccL \oplus \ccg_\ccL\).
Therefore \(\nsigma[\ybs](X - C) = \ccn \oplus \ccg = \cch\).

\end{proof}

 \begin{corollary}
 \comment{Cor. over}
\label{sec:orgheadline52}
\label{orgtarget53}
Let \(\sgC\) be a subset of \(\Nb^m\) satisfying {\normalfont{(\hyperref[orgtarget33]{SG2})}}.
Let \(X \in \Nb^m, n = \nsigma[\ybs](X), h \in \Nb\) with \(h \neq n\), and \(N = \max \set{L \in \NN : \ccn_L \neq \cch_L}\).
Let \(D \in \Nb^m\) with \(X_{\le N} - D \in \Nb^m\).
If \(D\) covers \(X\) at \(h\), then so does \(\sgC\).
 
\end{corollary}

\begin{proof}
 \comment{Proof.}
\label{sec:orgheadline53}
If \(h_N < n_N\), then \(h < n\), and hence \(\sgC\) covers \(X\) at \(h\) since \(\sgC\) satisfies {\normalfont{(\hyperref[orgtarget33]{SG2})}}.
Suppose that \(h_N > n_N\). We show that \(x^0_N + \cdots + x^{m - 1}_N \ge \ybs[N]\).
Let \(Y = X - D\).
Then
\begin{equation}
\label{orgtarget54}
 y^0_N \oplus \cdots \oplus y^{m - 1}_N  = h_N > n_N = x^0_N \oplus \cdots \oplus x^{m - 1}_N.
\end{equation}
Since \(X_{\le N} - D \in \Nb^m\), we see that \(y_{\le N}^i = x^i_{\le N} - d^i \le x^i_{\le N}\), and hence \(y^i_N \le x^i_N\). This implies that
\begin{equation}
\label{orgtarget55}
 y^0_N + \cdots + y^{m - 1}_N < x^0_N + \cdots + x^{m - 1}_N.
\end{equation}
Combining (\ref{orgtarget54}) and (\ref{orgtarget55}), we see that \(x^0_N + \cdots + x^{m - 1}_N \ge \ybs[N]\).
Lemma \ref{orgtarget52} implies that \(\sgC\) covers \(X\) at \(h\).

\end{proof}

\comment{connect}
\label{sec:orgheadline54}
The next result asserts that (\hyperref[orgtarget33]{SG2}) is a local property.

 \begin{lemma}
 \comment{Lem. local}
\label{sec:orgheadline55}
\label{orgtarget56}
Let \(\ccw = \min \set{m, \sup \set{\ybs[L] - 1 : L \in \NN}}\) and \(\sgC \subseteq \Nb^m\).
The following three conditions are equivalent.
\begin{enumerate}
\item \(\sgC\) satisfies {\normalfont{(\hyperref[orgtarget33]{SG2})}}.
\item \(\sgC|_{S}\) satisfies  {\normalfont{(\hyperref[orgtarget33]{SG2})}} for each \(S \subseteq \Omega\).
\item \(\sgC|_{S}\) satisfies {\normalfont{(\hyperref[orgtarget33]{SG2})}} for each \(S \in {\Omega \choose \ccw}\).
\end{enumerate}
 
\end{lemma}

\begin{proof}
 \comment{Proof.}
\label{sec:orgheadline56}
(1) \(\then\) (2).
We may assume that \(S = \set{0,1,\ldots, \cck - 1}\).
We show that \(\sgC|_S\) adequately covers every \(X' \in \Nb^\cck\).
Let \(n' = \nsigma[\ybs](X')\) and \(h' \in \Nb\) with \(h' < n'\).
Consider 
\[
 X = ((x')^0, \ldots, (x')^{\cck - 1}, 0, \ldots, 0) \in \Nb^m.
\]
Since \(\nsigma[\ybs](X) = n'\),
there exists \(C \in \sgC\) such that \(\nsigma[\ybs](X - C) = h'\) and
\[
 X - C = ((x')^0 - c^0, \ldots, (x')^{\cck - 1} - c^{\cck - 1}, - c^\cck, \ldots, - c^{m - 1}) \in \Nb^m.
\]
It follows that \(C = (c^0, \ldots, c^{\cck - 1}, 0, \ldots, 0)\),
and hence \(C|_S \in \sgC|_S, X' - C|_S \in \Nb^\cck\), and \(\nsigma[\ybs](X' - C|_S) = \nsigma[\ybs](X - C) = h'\).
Thus \(\sgC|_S\) adequately covers \(X'\).

(2) \(\then\) (3). This is trivial.

(3) \(\then\) (1).
If \(\ccw = m\), then this is trivial.
Suppose that \(\ccw < m\). Then \(\ccw = \sup \set{\ybs[L] - 1 : L \in \NN}\).
For \(X \in \Nb^m\), we show that \(\sgC\) adequately covers \(X\).
Let \(n = \nsigma[\ybs](X), h \in \Nb\) with \(h < n\), and
\(N = \max \set{L \in \NN : n_L \neq h_L}\).
By rearranging \(x^i\), we may assume that
\(x^0_N \ge x^1_N \ge \cdots \ge x^{m -1}_N\). 
Let \(S = \set{0, \ldots, \ccw - 1}\), \(\sgC' = \sgC |_{S}, X' = X|_{S}, n' = \nsigma[\ybs](X'), \ccl = n \ominus n' = x^{\ccw} \oplus \cdots \oplus x^{m - 1}\), and 
\(\cch' = \cch \ominus \ccl\).
Then 
\begin{equation}
\label{orgtarget57}
 \ccn = \ccn' \oplus \ccl \quad \tand \quad \cch = \cch' \oplus \ccl.
\end{equation}

We first show that \(\sgC'\) covers \(X'\) at \(h'\).
By (\ref{orgtarget57}),
\[
 \max \set{L \in \NN : n'_L \neq h'_L} = \max \set{L \in \NN : n_L \neq h_L} = N.
\]
If \(x^{\ccw - 1}_N \ge 1\), then \(x^0_N + \cdots + x^{\ccw - 1}_N \ge \ccw \ge \ybs[N] - 1\), so Lemma \ref{orgtarget52} implies that \(\sgC'\) covers \(X'\) at \(h'\).
Suppose that \(x^{\ccw - 1}_N  = 0\). Then \(h' < n'\).
Indeed, since \(x^{\ccw - 1}_N = x^{\ccw}_N = \cdots = x^{m - 1}_N = 0\), it follows that \(n'_N = n_N\),
and hence that \(h'_N = n'_N \ominus n_N \oplus h_N = h_N < n_N = n'_N\).
This implies that \(h' < n'\), so \(\sgC'\) covers \(X'\) at \(h'\), since \(\sgC'\) satisfies (\hyperref[orgtarget33]{SG2}).
Choose \(C' \in \sgC'\) so that \(C'\) covers \(X'\) at \(h'\).

We now prove that \(\sgC\) covers \(X\) at \(h\).
Let \(C = ((c')^0, \ldots, (c')^{\ccw - 1}, 0, \ldots, 0) \in \Nb^m\).
Then \(C \in \sgC\), \(X - C \in \Nb^m\), and
\begin{align*}
 \nsigma[\ybs](X - C) &= \nsigma[\ybs](X' - C') \oplus x^\ccw \oplus \cdots \oplus x^{m - 1} \\
 &= h' \oplus \ccl = \cch.
\end{align*}

\end{proof}

\comment{connect}
\label{sec:orgheadline57}
The third lemma allows us to construct a new Sprague-Grundy system from a given one by replacing a part.

Let \(S\) be a subset of \(\Omega\).
For \(C' \in \Nb^{|S|}\), let \(C'\uparrow^\Omega_S\) denote \(C \in \Nb^m\) defined by
\[
 c^i = \begin{cases}
 (c')^i & \tif i \in S, \\
 0 & \tif i \in \Omega \setminus S.
 \end{cases}
\]
For example, if \(\Omega = \set{0,1,2}, S = \set{0,2}\), and \(C' = (2,3)\), then 
\(C'\uparrow^\Omega_S = (2,0,3)\).

 \begin{lemma}
 \comment{Lem. replace}
\label{sec:orgheadline58}
\label{orgtarget58}
Let \(\sgC \in \Delta(\nsigma[\ybs, m])\), \(S \subseteq \Omega\), and \(\sgC' = \sgC|_S\).
Let \(\sgD' \in \Delta(\nsigma[\ybs, \size{S}])\), and consider
\begin{align*}
 \sgD  &= \left(\sgC \setminus \left(\sgC'\uparrow^\Omega_S \right) \right) \sqcup \left(\sgD'  \uparrow^\Omega_S \right),
\end{align*}
where \(\sgC'\uparrow^\Omega_S = \set{C' \uparrow^\Omega_S : C' \in \sgC'}\).
Then \(\sgD \in \Delta(\nsigma[\ybs, m])\).
Moreover, if \(\sgC\) and \(\sgD'\) are minimal systems of \(\nsigma[\ybs]\), then so is \(\sgD\).
 
\end{lemma}

\begin{proof}
 \comment{Proof.}
\label{sec:orgheadline59}
We may assume that \(S = \set{0, 1, \ldots, \cck - 1}\).

We first show that \(\sgD \in \Delta(\nsigma[\ybs, m])\).
It is easy to check that \(\sgD\) satisfies (\hyperref[orgtarget32]{SG1}).
We show that \(\sgD\) also satisfies (\hyperref[orgtarget33]{SG2}).
Let \(X \in \Nb^m, n = \nsigma[\ybs](X)\), and \(h \in \Nb\) with \(h < n\). Set \(N = \max \set{L \in \NN : n_L \neq h_L}\).
By Lemma \ref{orgtarget50}, 
there exists \(C \in \sgC\) such that \(\nsigma[\ybs](X - C) = h\) and
\begin{equation}
\label{orgtarget59}
 X_{\le N} - C \in \Nb^m.
\end{equation}
If \(C \in \sgC \setminus (\sgC' \uparrow^\Omega_S)\) (\(\subseteq \sgD\)), then \(\sgD\) covers \(X\) at \(h\).
Suppose that \(C \in \sgC' \uparrow^\Omega_S\).
Let
\[
 C' = C|_S, \quad X' = X|_S, \quad n' = \nsigma[\ybs]\left(X'\right), \tand h' = \nsigma[\ybs]\left(X' - C' \right).
\]
As we have seen in the proof of Lemma \ref{orgtarget58}, 
to prove that \(\sgD\) covers \(X\) at \(\cch\),
it suffices to show that \(\sgD'\) covers \(X'\) at \(\cch'\)
since \(\ccn \ominus \ccn' = \cch \ominus \cch' = x^k \oplus \cdots \oplus x^{m - 1}\) and
\[
 \sgD|_S = \Bigl(\left(\sgC \setminus \left(\sgC'\uparrow^\Omega_S \right) \right) \sqcup \left(\sgD'  \uparrow^\Omega_S \right)\Bigr)\Bigl|_S =  \left(\sgD'  \uparrow^\Omega_S \right)\bigl|_S = \sgD'.
\]
By (\ref{orgtarget59}), we see that \(X'_{\le N} - C' \in \Nb^\cck\). 
Since \(N = \max \set{L \in \NN : n'_L \neq h'_L}\), it follows from Corollary \ref{orgtarget53} that \(\sgD'\) covers \(X'\) at \(h'\).
Therefore \(\sgD\) satisfies (\hyperref[orgtarget33]{SG2}).

We next show the second part of the lemma.
Suppose that \(\sgC\) and \(\sgD'\) are minimal systems, and let \(\msim{\sgD}\) be a Sprague-Grundy system of \(\nsigma[\ybs]\) with \(\msim{\sgD} \subseteq \sgD\).
We show that \(\msim{\sgD} = \sgD\).
Since \(\msim{\sgD}|_S \subseteq \sgD|_S = \sgD'\) and \(\msim{\sgD}|_S\) is a Sprague-Grundy system of \(\nsigma[\ybs]\),
it follows from the minimality of \(\sgD'\) that \(\msim{\sgD}|_S = \sgD'\).
Hence it suffices to show that \(\msim{\sgD} \setminus (\sgD' \uparrow_S^\Omega) = \sgD \setminus (\sgD' \uparrow_S^\Omega)\).
Consider 
\[
 \msim{\sgC} = \left(\msim{\sgD} \setminus \left(\sgD' \uparrow_S^\Omega \right)\right) \sqcup \left(\sgC' \uparrow_S^\Omega \right).
\]
By the first part of the lemma, \(\msim{\sgC} \in \Delta(\nsigma[\ybs])\).
Since
\[
 \msim{\sgC}  \subseteq \left(\sgD \setminus \left(\sgD' \uparrow_S^\Omega \right)\right) \sqcup \left(\sgC' \uparrow_S^\Omega \right) = \left(\sgC \setminus \left(\sgC' \uparrow_S^\Omega \right)\right) \sqcup \left(\sgC' \uparrow_S^\Omega \right) = \sgC,
\]
it follows from the minimality of \(\sgC\) that \(\msim{\sgC} = \sgC\).
Therefore
\[
 \msim{\sgD} \setminus \left(\sgD' \uparrow_S^\Omega\right) = \msim{\sgC} \setminus \left(\sgC' \uparrow_S^\Omega\right) = \sgC \setminus \left(\sgC' \uparrow_S^\Omega\right) =  \sgD \setminus \left(\sgD' \uparrow_S^\Omega\right).
\]

\end{proof}

\section{Minimal systems}
\label{sec:orgheadline75}
\label{orgtarget30}
\subsection{Proof of Theorem \texorpdfstring{\ref{orgtarget3}}{1.7}}
\label{sec:orgheadline66}
Let \(\ccw = \min \Set{m, \sup \set{\ybs[L] - 1 : L \in \NN}}\).

\comment{Proof of (1)}
\label{sec:orgheadline62}
(1) We first show that \(\wt(\sgC) \ge \ccw\) for every \(\sgC \in \Delta(\nsigma[\ybs, m])\).
By the definition of \(\ccw\), we see that \(\ccw \le m\) and \(\ccw \le \ybs[N] - 1\) for some \(N \in \NN\).
Let
\[
 X = (\underbrace{\ybsp[N], \ldots, \ybsp[N]}_{\ccw}, 0, \ldots, 0) \in \Nb^m.
\]
Then \(\nsigma[\ybs](X) = \ccw \ybsp[N] > 0\).
Since \(\sgC \in \Delta(\nsigma[\ybs])\), there exists \(C \in \sgC\) such that
\(C\) covers \(X\) at 0.
Let \(Y = X - C\). Then \(y^i_N = 0\) for every \(i \in \Omega\) because
\(\nsigma[\ybs]<N>(Y) = y_N^0 \oplus \cdots \oplus y_N^{m - 1} = 0\) and \(y_N^0 + \cdots + y^{m - 1}_N \le x^0_N + \cdots + x^{m - 1}_N = \ccw \le \ybs[N] - 1\).
This implies that \(\wt(C) = \ccw\). 
Hence \(\wt(\sgC) \ge \ccw\).
It remains to show that \(\wt(\sgC) = \ccw\) for some \(\sgC \in \Delta(\nsigma[\ybs, m])\).

We prove that \(\wt(\sgCnmin[\ybs, m]) = \ccw\).
Let \(C \in \sgCnmin[\ybs, m]\) and \((N, j) \in Nj(C)\).
Since \(\wt(C) \le m\) and \(\min \set{m, \ybs[N] - 1} \le \min \set{m, \sup \set{\ybs[L] - 1 : L \in \NN}} = \ccw\),
we need only show that \(\wt(C) \le \ybs[N] - 1\).
By (\hyperref[orgtarget38]{C1}), if \(i \neq j\), then \(c^i \neq 0\) if and only if \(c^i_N \neq 0\).
Thus 
\[
 \wt(C) \le 1 + \sum_{i \neq j} c^i_N \le c^j_N + \delta(c^j_N) + \sum_{i \neq j} c^i_N \le \ybs[N] - 1
\]
by (\hyperref[orgtarget39]{C2}). This implies that \(\wt(C) \le \ccw\). Therefore \(\wt(\sgCnmin[\ybs, m]) = \ccw\).

\comment{Proof of (2).}
\label{sec:orgheadline63}
(2) 
(a) \(\then\) (b).
Let \(S \subseteq \Omega\) and \(\sgC' = \sgC|_S\). 
Then \(\sgC'\) satisfies (\hyperref[orgtarget32]{SG1}), so \(\sgC'\in \Delta(\nsigma[\ybs, \size{S}])\) by Lemma \ref{orgtarget56}.
We show that \(\sgC'\) is minimal.
Let \(\sgD'\) be a Sprague-Grundy system of \(\nsigma[\ybs, \size{S}]\) with \(\sgD' \subseteq \sgC'\), and consider
\[
 \sgD = \left(\sgC \setminus \left(\sgC' \uparrow_{S}^{\Omega}\right)\right) \sqcup \left(\sgD' \uparrow_{S}^{\Omega}\right) \text{(}\subseteq \sgC \text{)}.
\]
Lemma \ref{orgtarget58} implies that \(\sgD \in \Delta(\nsigma[\ybs, m])\).
Since \(\sgD \subseteq \sgC\), it follows from the minimality of \(\sgC\) that \(\sgD = \sgC\).
Hence \(\sgD' = \sgD|_S = \sgC|_S = \sgC'\), so \(\sgC'\) is a minimal system of \(\nsigma[\ybs, \size{S}]\).

(b) \(\then\) (c).
We need only show that \(\wt(\sgC) = \ccw\).
Let
\[
 \sgD = \Set{C \in \sgC : \wt(C) \le \ccw}.
\]
If \(S \in {\Omega \choose \ccw}\), then \(\sgD|_S = \sgC|_S \in \Delta(\nsigma[\ybs, \ccw])\).
Since \(\sgD\) satisfies (\hyperref[orgtarget32]{SG1}), it follows from Lemma \ref{orgtarget56} that \(\sgD \in \Delta(\nsigma[\ybs])\).
By the minimality of \(\sgC\), we see that \(\sgD = \sgC\).
Hence \(\wt(\sgC) = \wt(\sgD) \le \ccw\). By (1), \(\wt(\sgC) = \ccw\).

(c) \(\then\) (a).
It is easy to see that \(\sgC\) satisfies (\hyperref[orgtarget32]{SG1}).
Hence \(\sgC \in \Delta(\nsigma[\ybs, m])\) by Lemma \ref{orgtarget56}.
It remains to verify that \(\sgC\) is minimal.
Let \(\sgD\) be a Sprague-Grundy system of \(\nsigma[\ybs, m]\) with \(\sgD \subseteq \sgC\).
Since \(\wt(\sgC) = \wt(\sgD) = \ccw\), it follows that
\[
 \sgC = \bigcup_{S \in {\Omega \choose \ccw}} \left(\sgC|_S\right) \uparrow_S^\Omega \quad \tand \quad \sgD = \bigcup_{S \in {\Omega \choose \ccw}} \left(\sgD|_S\right) \uparrow_S^\Omega.
\]
Let \(S \in {\Omega \choose \ccw}\).
Since \(\sgD|_S \subseteq \sgC|_S\), 
it follows from the minimality of \(\sgC|_S\) that \(\sgD|_S = \sgC|_S\). Hence \(\sgD = \sgC\).
This means that \(\sgC\) is a minimal system.

\comment{Proof of (3).}
\label{sec:orgheadline64}
(3) (\(\Leftarrow\)). We show that 
\(\sgCnmin[\ybs]\) is a unique minimal system of \(\nsigma[\ybs]\) when \(\ybs = (\ybs[0], 2, 2, \ldots)\).

We first verify that
\begin{equation}
\label{orgtarget60}
 \sgCnmin[\ybs] = \Set{C \in \Nb^m : \wt(C) = 1} \cup \Set{C \in \Nb^m : 1 \le \sum_{i \in \Omega} c^i \le \ybs[0] - 1}.
\end{equation}
Let \(\sgD\) be the right-hand side of (\ref{orgtarget60}).
By the definition of \(\sgCnmin[\ybs]\), we see that \(\sgD \subseteq \sgCnmin[\ybs]\).
Let \(C \in \sgCnmin[\ybs]\) and \((N, j) \in Nj(C)\). Then \(c^j \neq 0\).
We divide the proof into two cases.
First, suppose that \(N = 0\). From (\hyperref[orgtarget38]{C1}), we see that \(c^i = c^i_0\) for every \(i \in \Omega\).
Hence \(1 \le \sum c^i \le \ybs[0] - 1\) by (\hyperref[orgtarget39]{C2}). This means that \(C \in \sgD\).
Second, suppose that \(N \ge 1\).
Then
\[
 \sum_{i \neq j} c^i_N < \sum_i c^i_N + \delta(c_N^j) \le  \ybs[N] - 1 = 1,
\]
by (\hyperref[orgtarget39]{C2}).
Thus if \(i \neq j\), then \(c^i_N = 0\), and hence \(c^i = 0\) by (\hyperref[orgtarget38]{C1}).
In particular, \(\wt(C) = 1\). Therefore \(\sgCnmin[\ybs] \subseteq \sgD\).

We now show that \(\sgCnmin[\ybs]\) is a unique minimal system of \(\nsigma[\ybs]\).
It suffices to show that \(\sgCnmin[\ybs] \subseteq \sgC\) for every \(\sgC \in \Delta(\nsigma[\ybs])\).
Let \(X \in \sgCnmin[\ybs]\).
It follows from (\ref{orgtarget60}) that \(\nsigma[\ybs](X) = \sum_i x^i > 0\).
Hence there exists \(C \in \sgC\) such that \(\nsigma[\ybs](X - C) = 0\) and \(X - C \in \Nb^m\).
By (\ref{orgtarget60}), we see that \(X - C = (0, \ldots, 0)\),
and hence \(X = C \in \sgC\). We conclude that \(\sgC \supseteq \sgCnmin[\ybs]\).

(\(\then\)).
We prove that \(\nsigma[\ybs]\) has at least two minimal systems when \(\ybs[N] \ge 3\) for some \(N \ge 1\).
Lemma \ref{orgtarget58} implies that it suffices to show the assertion for \(m = 2\).
Let \(C^{(0)} = (1, \ybsp[N]), C^{(1)} = (\ybsp[N], 1)\), and \(\sgB = \sgCnmin[\ybs] \setminus \set{C^{(0)}, C^{(1)}}\).
Then \(C^{(0)} \neq C^{(1)}\) since \(N \ge 1\).
We show the following two assertions:
\begin{enumerate}[label={(\alph*)}]
\item \(\sgB \cup \set{C^{(i)}} \in \Delta(\nsigma[\ybs])\) for each \(i \in \set{0, 1}\).
\item \(\sgB \not \in \Delta(\nsigma[\ybs])\).
\end{enumerate}
By (a) and (b), \(\nsigma[\ybs]\) has at least two minimal systems.
Indeed, \(\sgB \cup \set{C^{(i)}}\) contains a minimal system \(\sgC^{(i)}\) by (a).
From (b), we see that \(C^{(i)}\) must be in \(\sgC^{(i)}\), so \(\sgC^{(0)} \neq \sgC^{(1)}\).

(a)
It suffices to show the assertion for \(i = 0\).
Let \(X \in \Nb^2, n = \nsigma[\ybs](X)\), and \(h \in \Nb\) with \(h < n\). Set \(\msim{\ccN} = \max \set{L \in \NN : n_L \neq h_L}\).
We show that \(\sgB \cup \set{C^{(0)}}\) covers \(X\) at \(h\).
Choose \(j \in \set{0, 1}\) so that \(x^j_{\msim{\ccN}} \ge x^0_{\msim{\ccN}} , x^1_{\msim{\ccN}}\).
Lemma \ref{orgtarget42} shows that there exists \(C \in \sgCnmin[\ybs]\) satisfying (\hyperref[orgtarget46]{A1}) \(\nsigma[\ybs](X - C) = h\);
(\hyperref[orgtarget47]{A2}) \((\msim{\ccN}, j) \in Nj(C)\); (\hyperref[orgtarget48]{A3}) \(X_{\le \msim{\ccN}} - C \in \Nb^2\); and (\hyperref[orgtarget49]{A4}) \((x^j - c^j)_{\msim{\ccN}} < x^j_{\msim{\ccN}}\).
If \(C \neq C^{(1)}\), then \(C \in \sgB \cup \set{C^{(0)}}\).
Suppose that \(C = C^{(1)}\). 
Since \(Nj(C^{(1)}) = \set{(N, 1)}\), it follows from (\hyperref[orgtarget47]{A2}) that \((\msim{\ccN}, j) = (N, 1)\). 
By (\hyperref[orgtarget49]{A4}), \((x^1 - 1)_N < x^1_N\).
This implies that \(x^1_{\le N} = x^1_N \ybsp[N] = \bnum{0, \ldots, 0, x^1_N}\).
By (\hyperref[orgtarget48]{A3}),
\[
X - C^{(1)} =\begin{bmatrix}
 x^0_0 & \cdots & x^0_{N - 1} & x^0_N  \ominus 1 & x^0_{N + 1} & \cdots \\
 \ominus 1 & \cdots & \ominus 1 & x^1_N  \ominus 1 & x^1_{N + 1} & \cdots \\
\end{bmatrix},
\]and hence by (\hyperref[orgtarget46]{A1}),
\[
 h = \nsigma[\ybs](X - C^{(1)}) = \bnum{x^0_0 \ominus 1, \ldots, x^{0}_{N - 1} \ominus 1, x^0_N \oplus x^1_N \ominus 2, x^0_{N+1} \oplus x^1_{N+1}, \ldots}.
\]
We divide the proof into two cases. First, suppose that \(x^1_N \ge 2\). Consider \(D = (0, 1 + \ybsp[N]) \in \sgB\). Then \(X - D \in \Nb^2\) and \(\nsigma[\ybs](X - D) = h\).
Hence \(D\) covers \(X\) at \(h\).
Next, suppose that \(x^1_N = 1\). Then \(x^1_{\le N} = \ybsp[N]\). Since \(x^0_{\le N }\le x^1_{\le N}\), it follows from (\hyperref[orgtarget48]{A3}) that 
\[
 c^0 = \ybsp[N] \le x^0_{\le N} \le x^1_{\le N} = \ybsp[N],
\]
so  \(x^0_{\le N} = x^1_{\le N} = \ybsp[N]\). 
Thus \(\nsigma[\ybs](X - C^{(0)}) = \nsigma[\ybs](X - C^{(1)}) = h\), and hence \(C^{(0)}\) covers \(X\) at \(h\).
We conclude that \(\sgB \cup \set{C^{(0)}} \in \Delta(\nsigma[\ybs])\). \footnote{Observe that \(\sgB\) adequately covers \(X\) with \(X_{\le N} \neq (\ybsp[N], \ybsp[N])\). We use this fact in the proof of Theorem \ref{orgtarget19}.}

(b)
Consider \(X = (\ybsp[N], \ybsp[N])\).
Then \(\nsigma[\ybs](X - C^{(0)}) = \nsigma[\ybs](X - C^{(1)}) = \ybsp[N] - 1\).
We show that \(\sgB\) does not cover \(X\) at \(\ybsp[N] - 1\).
Suppose that \(C \in \sgCnmin[\ybs]\) covers \(X\) at \(\ybsp[N] - 1\).
It suffices to show that \(C \in \set{C^{(0)}, C^{(1)}}\).
Since \(C \in \sgCnmin[\ybs]\), we see that \((\msim{\ccN}, j) \in Nj(C)\) for some \(\msim{\ccN}, j \in \NN\).
Suppose that \(j = 0\). Then
\begin{equation}
\label{orgtarget61}
C = \begin{bmatrix}
c^0_0 & \ldots & c^0_{\msim{\ccN} - 1} & c^0_{\msim{\ccN}} \\
0 & \ldots & 0 &  c^1_{\msim{\ccN}}
\end{bmatrix} \tand c^0_{\msim{\ccN}} + c^1_{\msim{\ccN}} \le \ybs[\msim{\ccN}] - 1.
\end{equation} 
Let \(Y = X - C\).
Since \(\nsigma[\ybs](Y) = \ybsp[N] - 1\),
\begin{equation}
\label{orgtarget62}
 y^0_L \oplus y^1_L = \begin{cases} 
 \ominus 1 & \tif 0 \le L \le N - 1, \\
 0 & \tif L \ge N.
 \end{cases}
\end{equation}
If \(y^0_L \oplus y^1_L = \ominus 1\), then \(y^0_L + y^1_L = \ominus 1\).
Hence \(y^0 + y^1 = \ybsp[N] - 1\).
We also see that \(c^0, c^1 \neq 0\) and \(c^0 + c^1 = x^0 + x^1 - (y^0 + y^1) = \ybsp[N] + 1\).
By (\ref{orgtarget61}), \(c^0 = 1\) and \(c^1 = \ybsp[N]\),
that is, \(C = C^{(0)}\). By symmetry, if \((\msim{\ccN}, 1) \in Nj(C)\), then \(C = C^{(1)}\).
We conclude that \(\sgB\) is not a Sprague-Grundy system of \(\nsigma[\ybs]\).

\qed

 \begin{example}
 \comment{Exm. min-system}
\label{sec:orgheadline65}
\label{orgtarget18}
Let \(\ybs = (2,3,2,2,\ldots)\) and \(m = 2\).
We give two minimal systems of \(\nsigma[\ybs, 2]\).
A straightforward computation shows that
\[
 \sgCnmin[\ybs, 2] = \set{C \in \Nb^2 : \wt(C) = 1} \cup \set{C^{(0)}, C^{(1)}, D,  E^{(0)}, E^{(1)}},
\]
where
\[
C^{(0)} =\begin{bmatrix}
 1 & 0 \\
 0 & 1 \\
\end{bmatrix},
C^{(1)} =\begin{bmatrix}
 0 & 1 \\
 1 & 0 \\
\end{bmatrix},
D =\begin{bmatrix}
 0 & 1 \\
 0 & 1 \\
\end{bmatrix},
E^{(0)} =\begin{bmatrix}
 1 & 1 \\
 0 & 1 \\
\end{bmatrix},
E^{(1)} =\begin{bmatrix}
 0 & 1 \\
 1 & 1 \\
\end{bmatrix}.
\]Consider
\[
 \sgC^{(0)} = \sgCnmin[\ybs, 2] \setminus \set{C^{(1)}, E^{(1)}} \tand \sgC^{(1)} = \sgCnmin[\ybs, 2] \setminus \set{C^{(0)}, E^{(0)}}.
\]
We verify that \(\sgC^{(0)}\) and \(\sgC^{(1)}\) are minimal systems of \(\nsigma[\ybs, 2]\).
By symmetry, it suffices to show the assertion for \(\sgC^{(0)}\).

We first show that \(\sgC^{(0)} \in \Delta(\nsigma[\ybs])\).
Let \(X \in \Nb^2, n = \nsigma[\ybs](X)\), and \(h \in \Nb\) with \(h < n\). Set \(N = \max \set{L \in \NN: n_L \neq h_L}\).
We show that \(\sgC^{(0)}\) covers \(X\) at \(h\).
Choose \(j \in \set{0,1}\) so that \(x^j_N \ge x^0_N, x^1_N\).
By Lemma \ref{orgtarget42}, there exists \(C \in \sgCnmin[\ybs]\) satisfying (\hyperref[orgtarget46]{A1})-(\hyperref[orgtarget49]{A4}).
If \(C \not \in \set{C^{(1)}, E^{(1)}}\), then \(C \in \sgC^{(0)}\), so
we may assume that \(C \in \set{C^{(1)}, E^{(1)}}\).

If \(C = C^{(1)}\), then \(\set{C^{(0)}, (0, 1 + \ybsp[1])}\) covers \(X\) at \(h\) as we have seen in the proof of Theorem \ref{orgtarget3},
and hence so does \(\sgC^{(0)}\) since \(\set{C^{(0)}, (0, 1 + \ybsp[1])} \subset \sgC^{(0)}\).

Suppose that \(C = E^{(1)}\). Then \(Nj(E^{(1)}) = \set{(1, 1)}\), so \((N, j) = (1, 1)\) and \(n_1 \neq h_1\).
Since \(\ybs[1] = 3\) and
\[
X_{\le 1} - E^{(1)} =\begin{bmatrix}
 x^0_0 & x^0_1 \\
 x^1_0 & x^1_1 \\
\end{bmatrix}
-\begin{bmatrix}
 0 & 1 \\
 1 & 1 \\
\end{bmatrix}=
\begin{bmatrix}
 x^0_0 & x^0_1  \ominus 1 \\
 x^1_0  \ominus 1 & x^1_1  \ominus 1 \ominus \delta(x^1_0) \\
\end{bmatrix},
\]it follows that \(\delta(x^1_0) = 0\), for otherwise \(h_1 = x^0_1 \oplus x^1_1 \ominus 1 \ominus 1 \ominus 1 = x^0_1 \oplus x^1_1 = n_1\).
Consider
\[
\ccB =\begin{bmatrix}
 1 & 1 \ominus \delta(x^0_0) \\
 0 & 1 \\
\end{bmatrix}.
\]Observe that \(B \in \set{C^{(0)}, E^{(0)}}\) and \(X_{\le 1} - \ccB \in \Nb^2\). Moreover,
\[
X_{\le 1} - \ccB =\begin{bmatrix}
 x^0_0 & x^0_1 \\
 x^1_0 & x^1_1 \\
\end{bmatrix}
-\begin{bmatrix}
 1 & 1 \ominus \delta(x^0_0) \\
 0 & 1 \\
\end{bmatrix}=
\begin{bmatrix}
 x^0_0  \ominus 1 & x^0_1  \ominus 1 \\
 x^1_0 & x^1_1  \ominus 1 \\
\end{bmatrix}.
\]Hence \(\nsigma[\ybs](X - \ccB) = \nsigma[\ybs](X - E^{(1)})\).
Since \(\ccB \in \set{C^{(0)}, E^{(0)}}\), the set \(\sgC^{(0)}\) covers \(X\) at \(h\),
and therefore it is a Sprague-Grundy system of \(\nsigma[\ybs]\).

We next prove the minimality of \(\sgC^{(0)}\).
Let \(\sgC\) be a Sprague-Grundy system of \(\nsigma[\ybs]\).
It suffices to show the following three assertions:

\begin{enumerate}
\item \(\sgC\) contains \(\ccC^{(0)}\) or \(\ccC^{(1)}\).
\item \(\sgC\) contains \(\ccD\).
\item \(\sgC\) contains \(\ccE^{(0)}\) or \(\ccE^{(1)}\).
\end{enumerate}

First, we show (1) and (3).
Let \(\epsilon \in \set{0,1}\) and \(X = (\bnum{\epsilon, 1}, \bnum{\epsilon, 1})\).
Then \(\nsigma[\ybs](X) = \bnum{0,2}\), so \(X\) is covered at \(1\) by some \(C \in \sgC\).
Since \(X = (\bnum{\epsilon,1}, \bnum{\epsilon,1})\), we see that \(X - C  \in \set{(0, 1), (1, 0)}\).
Hence 
\[
 C \in \set{X - (0,1), X - (1,0)} = \begin{cases}
  \set{C^{(1)},  C^{(0)}} & \tif \epsilon = 0,\\
  \set{E^{(0)},  E^{(1)}} & \tif \epsilon = 1.\\
 \end{cases}
\]

Next, we verify (2).
Let \(X = \ccD\).
Then \(X\) is covered at \(0\) by some \(C \in \sgC\).
Since \(X = (\bnum{0,1}, \bnum{0,1})\), it follows that \(X - C  \in \set{(0, 0), (1, 1)}\).
Assume that \(X - C = (1,1)\).
Then \(C = (1,1)\), but \((1,1)\) does not satisfy (\hyperref[orgtarget32]{SG1}), because it covers \((1,1)\) at \(0\) and \(\nsigma[\ybs]((1,1)) = 0\).
This implies that \(C = X - (0,0) = D\).

Therefore \(\sgC^{(0)}\) and \(\sgC^{(1)}\) are minimal systems of \(\nsigma[\ybs]\).
Moreover, minimal systems of \(\nsigma[\ybs]\) are only \(\sgC^{(0)}\) and \(\sgC^{(1)}\).
Indeed, let \(\sgC\) be a Sprague-Grundy system of \(\nsigma[\ybs]\)
and \(X = E^{(0)}\). Then \(\nsigma[\ybs](X) = \bnum{1,2}\), so \(X\) is covered at \(0\) by some \(C \in \sgC\).
It follows that \(X - C  \in \set{(0, 0), (1, 1)}\).
Hence \(C \in \set{X - (0,0), X - (1,1)} = \set{E^{(0)},  C^{(1)}}\).
Therefore \(\sgC\) contains \(E^{(0)}\) or \(C^{(1)}\).
By symmetry, \(\sgC\) contains \(E^{(1)}\) or \(C^{(0)}\).
This implies that minimal systems of \(\nsigma[\ybs]\) are only \(\sgC^{(0)}\) and \(\sgC^{(1)}\).
 
\end{example}

\subsection{Proof of Theorem \texorpdfstring{\ref{orgtarget19}}{1.9}}
\label{sec:orgheadline74}
\comment{unique minimal}
\label{sec:orgheadline67}
We first show that \(\sgCnmin[\ybs]\) is a unique minimal symmetric system of \(\nsigma[\ybs]\) when \(\ybs = (\ybs[0],2,\ldots)\) or \(\ybs = (2,3,2,\ldots)\).
It is obvious that \(\sgCnmin[\ybs]\) is symmetric.
We show the minimality of \(\sgCnmin[\ybs]\). Let \(\sgC\) be a symmetric Sprague-Grundy system of \(\nsigma[\ybs]\).
If \(\ybs = (\ybs[0], 2, \ldots)\), then \(\sgCnmin[\ybs] \subseteq \sgC\) since \(\sgCnmin[\ybs]\) is the minimal system of \(\nsigma[\ybs]\).
Suppose that \(\ybs = (2, 3, 2, \ldots)\).
Since, for \(S \in {\Omega \choose 2}\), the set \(\sgC|_S\) is a symmetric Sprague-Grundy system of \(\nsigma[\ybs, 2]\),
it follows from Example \ref{orgtarget18} that \(\sgCnmin[\ybs, 2] \subseteq \sgC|_S\).
Since \(\wt(\sgCnmin[\ybs, m]) = 2\), we see that
\[
 \sgCnmin[\ybs, m] = \bigcup_{S \in {\Omega \choose 2}} \sgCnmin[\ybs, 2] \uparrow_{S}^{\Omega} \ \ \ \subseteq \ \bigcup_{S \in {\Omega \choose 2}} \left(\sgC|_S \right) \uparrow_{S}^{\Omega} \ \ \ \subseteq \ \sgC.
\]

\comment{non unique minimal}
\label{sec:orgheadline73}
We next show that the function \(\nsigma[\ybs]\) has at least two minimal symmetric systems
when \(\ybs \neq (\ybs[0], 2, \ldots)\) and \(\ybs \neq (2,3,2,\ldots)\).
 
 \begin{step}
 \comment{Step.}
\label{sec:orgheadline68}
Let \(\sgC\) and \(\sgD'\) be symmetric systems of \(\nsigma[\ybs, m]\) and \(\nsigma[\ybs, k]\) with \(k \le m\), respectively.
Let \(\sgC' = \sgC|_{\set{0,1,\ldots, k - 1}}\).
Consider
\begin{equation}
\label{orgtarget63}
 \sgD = \sgC[\sgD'] = \bigg(\sgC \setminus \bigcup_{S \in {\Omega \choose \cck}} \big(\sgC' \uparrow_S^\Omega \big) \bigg) \sqcup \bigcup_{S \in {\Omega \choose \cck}} \big(\sgD' \uparrow_S^\Omega \big).
\end{equation}
We show that \(\sgD\) is a symmetric system of \(\nsigma[\ybs, m]\) with \(\sgD|_T = \sgD'\) for every \(T \in {\Omega \choose \cck}\).

We first verify that \(\sgD|_T = \sgD'\) for \(T \in {\Omega \choose \cck}\).
Since
\[
 \sgD|_T = \bigcup_{S \in {\Omega \choose \cck}} \big(\sgD' \uparrow_S^\Omega \big) \big|_T,
\]
it follows that \(\sgD|_T \supseteq \sgD'\).
We show that \(\sgD|_T \subseteq \sgD'\).
Let \(D' \in \sgD|_T\). Then \(D' \in (\sgD' \uparrow_S^\Omega)|_T\) for some \(S \in {\Omega \choose k}\),
and hence \(D' = (E' \uparrow_S^\Omega)|_T\) for some \(E' \in \sgD'\).
This implies that \(D'\) can be obtained by permuting the coordinates of \(E\), that is,
\(D' = ((e')^{\pi(0)}, \ldots (e')^{\pi(\cck - 1)})\) for some \(\pi \in \fS_{k}\). \footnote{For example, let \(\Omega = \{0,1,2,3\}, T = \{0,1,2\}, S = \{0,2,3\}\), and \((a,b,0) \in \sgD'\).
Then \(((a,b,0) \uparrow^\Omega_S)|_T = (a,0,b,0)|_T = (a,0,b)\). Since \(\sgD'\) is symmetric, we see that \((a,0,b) \in \sgD'\).}
It follows from the symmetry of \(\sgD'\) that \(D' \in \sgD'\).
Hence \((\sgD' \uparrow_S^\Omega)|_T \subseteq \sgD'\). This implies that \(\sgD|_T \subseteq \sgD'\).

We next show that \(\sgD\) is a symmetric system.
Since \(\sgC\) and \(\sgD'\) satisfy (\hyperref[orgtarget32]{SG1}), so does \(\sgD\).
We show that \(\sgD\) also satisfies (\hyperref[orgtarget33]{SG2}).
Let \(X \in \Nb^m, n = \nsigma[\ybs](X)\), and \(h \in \Nb\) with \(h < n\).
Set \(N = \max \set{L \in \NN : n_L \neq h_L}\).
By Lemma \ref{orgtarget50},
we see that \(X\) is covered at \(h\) by some \(C \in \sgC\) with \(X_{\le N} - C \in \Nb^m\).
If 
\[
 C \in \sgC \setminus \bigcup_{S \in {\Omega \choose k}} \big(\sgC' \uparrow_S^\Omega \big),
\]
then \(C \in \sgD\), so \(\sgD\) covers \(X\) at \(h\).
Suppose that \(C \in \bigcup (\sgC' \uparrow_S^\Omega)\), and choose \(S \in {\Omega \choose k}\) such that \(C \in \sgC' \uparrow_S^\Omega\).
Let \(C' = C|_S, X' = X|_S, n' = \nsigma[\ybs](X')\), and \(h' = \nsigma[\ybs](X' - C')\).
Then \(\ccn \ominus \ccn' = \cch \ominus \cch'\).
Since \(N = \max \set{L \in \NN : n'_L \neq h'_L}\) and \(X'_{\le N} - C' \in \Nb^k\), it follows from Corollary \ref{orgtarget53} that
\(\sgD'\) covers \(X'\) at \(h'\). This implies that \(\sgD\) covers \(X\) at \(h\).
Since \(\sgD\) is symmetric, it is a symmetric system of \(\nsigma[\ybs, m]\).
 
\end{step}

 \begin{step}
 \comment{Step.}
\label{sec:orgheadline69}
We show that if \(\nsigma[\ybs, 2]\) has at least two minimal symmetric systems, then so does \(\nsigma[\ybs, m]\).
Let \(\sgC\) be a minimal symmetric system of \(\nsigma[\ybs, m]\) and \(\sgC'  = \sgC|_{\set{0,1}}\).

We first verify that \(\sgC'\) is a minimal symmetric system of \(\nsigma[\ybs, 2]\). \footnote{We can obtain an analogue of Theorem \ref{orgtarget3} (2) for minimal symmetric systems.}
Let \(\sgD'\) be a symmetric system with \(\sgD' \subseteq \sgC'\), and let \(\sgD = \sgC[\sgD']\) defined in (\ref{orgtarget63}).
By Step 1, we see that \(\sgD\) is a symmetric system with \(\sgD \subseteq \sgC\). It follows from the minimality of \(\sgC\) that \(\sgD = \sgC\).
Hence \(\sgD' = \sgD_{\set{0,1}} = \sgC_{\set{0,1}} = \sgC'\). Thus \(\sgC'\) is a minimal symmetric system.

We now show that \(\nsigma[\ybs, m]\) has at least two minimal symmetric systems.
By assumption, \(\nsigma[\ybs, 2]\) has a minimal symmetric system \(\sgD'\) with \(\sgD' \not \supseteq \sgC'\).
Let \(\sgD = \sgC[\sgD']\).
By Step 1, \(\sgD\) is a symmetric system, and therefore it contains a minimal symmetric system \(\msim{\sgD}\). \footnote{We can show that \(\msim{\sgD} = \sgD\) in the same way as in the proof of Lemma \ref{orgtarget58}.}
Since \(\sgD_{\set{0,1}} = \sgD' \not \supseteq \sgC'\), we conclude that \(\msim{\sgD} \neq \sgC\).
Therefore \(\nsigma[\ybs, m]\) has at least two minimal symmetric systems.
 
\end{step} 

\comment{connect}
\label{sec:orgheadline70}
By Step 2, we may assume that \(m = 2\).
Choose \(N \ge 1\) so that \(\ybs[N] \ge 3\).

 \begin{step}
 \comment{Step.}
\label{sec:orgheadline71}
\label{orgtarget64}
We show that \(\nsigma[\ybs]\) has at least two minimal symmetric systems when \(\ybs[0] \ge 3\) or \(N \ge 2\).
Since \(\sgCnmin[\ybs]\) is symmetric, it contains a minimal symmetric system \(\sgC\).
Let \(C = (1, \ybsp[N])\).
As we have seen in the proof of Theorem \ref{orgtarget3}, \(C\) must be in \(\sgC\).
Thus it suffices to find a symmetric system \(\sgD\) with \(C \not \in \sgD\).
Let 
\[
 \sgB = \sgCnmin[\ybs] \setminus \fS_2(C), \quad \ccD = (2, \ybsp[N] - 1), \ \tand \ \sgD = \sgB \cup \fS_2(\ccD).
\] 
Note that \(D \not \in \fS_2(C)\) since \(\ybs[0] \ge 3\) or \(N \ge 2\).\footnote{If \(\ybs[0] = 2\) and \(N = 1\), then \(\ccC = (1,2)\) and \(\ccD = (2,1)\), so \(D \in \fS_2(C)\).}
In particular, \(C \not \in \sgD\). We show that \(\sgD\) is a symmetric system.
As we have seen in the proof of Theorem \ref{orgtarget3}, \(\sgB\) adequately covers \(X\) with \(X_{\le N} \neq (\ybsp[N], \ybsp[N])\)
Hence we need only show the following two statements:
\begin{enumerate}
\item \(\nsigma[\ybs](X - \ccD) = \nsigma[\ybs](X - C)\) for every \(X\) with \(X_{\le N} = (\ybsp[N], \ybsp[N])\).
\item \(\ccD \in \sgCord[\ybs]\).\footnote{From (2), we see that \(\sgD\) satisfies (\hyperref[orgtarget32]{SG1}).}
\end{enumerate}
Let \(n = \nsigma[\ybs](X)\) and \(h = \nsigma[\ybs](X - C)\).
Note that \(h_{\le N} = \ybsp[N] - 1\) and \(h_{\ge N + 1} = n_{\ge N + 1}\) since \(X_{\le N} - C = (\ybsp[N], \ybsp[N]) - (1, \ybsp[N]) = (\ybsp[N] - 1, 0)\).

First, suppose that \(\ybs[0] \ge 3\). Then \(\ccD = (\bnum{2}, \bnum{\ominus 1, \ldots, \ominus 1}) \in \sgCord[\ybs]\).
Since
 \(X_{\le N} - \ccD = (\ybs[N] - 2, 1) = (\bnum{\ominus 2, \ominus 1, \ldots, \ominus 1}, \bnum{1})\),
it follows that \(\nsigma[\ybs](X - \ccD) = h\).

Next, suppose that \(\ybs[0] = 2\) and \(N \ge 2\). 
Then 
\[
 \ccD = (\bnum{0,1},\ [1, \underbrace{\ominus 1, \ldots, \ominus 1}_{N - 1}]) \in \sgCord[\ybs].
\]
Since \(X_{\le N} - \ccD = (\ybsp[N] - 2, 1) = (\bnum{0, \ominus 1, \ominus 1, \ldots, \ominus 1}, \bnum{1})\), we see that \(\nsigma[\ybs](X - \ccD) = h\).
Therefore \(\sgD\) is a symmetric system with \(C \not \in \sgD\).
 
\end{step}

 \begin{step}
 \comment{Step.}
\label{sec:orgheadline72}
It remains to show the assertion when \(\ybs = (2,\ybs[1],2,2,\ldots)\) with \(\ybs[1] \ge 4\).
Let
\[
C^{(0)} =\begin{bmatrix}
 1 & \ominus 2 \\
 0 & 1 \\
\end{bmatrix}\hspace{-0.3em},\
C^{(1)} =\begin{bmatrix}
 0 & 1 \\
 1 & \ominus 2 \\
\end{bmatrix}\hspace{-0.3em},\
\]%
\[
D^{(0)} =\begin{bmatrix}
 1 & 1 \\
 0 & \ominus 2 \\
\end{bmatrix}\hspace{-0.3em},\
D^{(1)} =\begin{bmatrix}
 0 & \ominus 2 \\
 1 & 1 \\
\end{bmatrix}\hspace{-0.3em},\
\ccE^{(0)} =\begin{bmatrix}
 1 & 0 \\
 0 & \ominus 2 \\
\end{bmatrix}\hspace{-0.3em},\
\ccE^{(1)} =\begin{bmatrix}
 0 & \ominus 2 \\
 1 & 0 \\
\end{bmatrix}\hspace{-0.3em}.
\]Note that \(C^{(i)} \neq D^{(i)}\) and \(E^{(i)} \in \sgCord[\ybs] \setminus \sgCnmin[\ybs]\) for each \(i \in \set{0, 1}\) since \(\ybs[1] \ge 4\).
Consider
\[
 \sgD = \sgCnmin[\ybs] \cup \Set{\ccD^{(0)}, \ccD^{(1)}, \ccE^{(0)}, \ccE^{(1)}} \setminus \Set{\ccC^{(0)}, \ccC^{(1)}}.
\]
It suffices to show the following two assertions:
\begin{enumerate}
\item Every Sprague-Grundy system of \(\nsigma[\ybs]\) contains \(\ccC^{(0)}\) or \(\ccE^{(1)}\).
\item \(\sgD\) is a symmetric Sprague-Grundy system of \(\nsigma[\ybs]\).
\end{enumerate}
By (1) and (2), we see that \(\sgCnmin[\ybs]\) and \(\sgD\) contain different minimal symmetric systems of \(\nsigma[\ybs]\).

We first show (1). Let \(X = \ccC^{(0)}\). 
Then \(\nsigma[\ybs](X) = \bnum{1, \ominus 1}\).
If \(Y\) is a descendant of \(X\) with \(\nsigma[\ybs](Y) = 0\), then
\(Y \in \set{(0, 0), (1, 1)}\).
Since \(X - (0, 0) = \ccC^{(0)}\) and \(X - (1, 1) = \ccE^{(1)}\), (1) holds.

We next verify (2). Since \(\sgD \subseteq \sgCord[\ybs]\), we see that \(\sgD\) satisfies (\hyperref[orgtarget32]{SG1}).
Let \(X \in \Nb^2, n = \nsigma[\ybs](X)\), and \(h \in \Nb\) with \(h < n\).
We show that \(\sgD\) covers \(X\) at \(h\).
Set \(\msim{N} = \max \set{L \in \NN : n_L \neq h_L}\).
By Lemma \ref{orgtarget42}, there exists \(C \in \sgCnmin[\ybs]\) satisfying (\hyperref[orgtarget46]{A1})-(\hyperref[orgtarget49]{A4}).
We may assume that \(C \in \set{C^{(0)}, C^{(1)}}\) since otherwise \(C \in \sgD\).
Suppose that \(C = C^{(0)}\).
Since \(N(C^{(0)}) = \set{1}\), it follows that \(\msim{N} = 1\) and \(n_1 \neq h_1\).
Because
\[
X_{\le 1} - C^{(0)} =\begin{bmatrix}
 x^0_0 & x^0_1 \\
 x^1_0 & x^1_1 \\
\end{bmatrix}
-\begin{bmatrix}
 1 & \ominus 2 \\
 0 & 1 \\
\end{bmatrix}=
\begin{bmatrix}
 x^0_0  \ominus 1 & x^0_1  \oplus  2 \ominus \delta(x^0_0) \\
 x^1_0 & x^1_1  \ominus 1 \\
\end{bmatrix},
\]we see that
\[
 h_{\le 1} = \nsigma[\ybs](X_{\le 1} - C^{(0)}) = 
 \bnum{x^0_0 \oplus x^1_0 \ominus 1, x^0_1 \oplus x^1_1 \oplus 1 \ominus \delta(x^0_0)}.
\]
Since \(h_1 \neq n_1\),  it follows that \(\delta(x^0_0) = 0\).
Consider
\[
D =\begin{bmatrix}
 0 & \ominus  2 \\
 1 & 1 \ominus  \delta(x^1_0) \\
\end{bmatrix}\in \set{\ccD^{(1)}, \ccE^{(1)}} (\subseteq \sgD).
\]Then
\[
X_{\le 1} - D  =\begin{bmatrix}
 x^0_0 & x^0_1  \oplus  2 \\
 x^1_0  \ominus 1 & x^1_1  \ominus 1 \\
\end{bmatrix}\in \Nb^2.
\]It follows that \(\nsigma[\ybs](X - D) = \nsigma[\ybs](X - C) = h\). This means that \(\sgD\) covers \(X\) at \(h\).
By symmetry, \(\sgD\) covers \(X\) at \(h\) even when \(C = C^{(1)}\).
Therefore \(\sgD\) is a symmetric Sprague-Grundy system of \(\nsigma[\ybs]\). 

\qed
 
\end{step} 

\section{Maximum systems}
\label{sec:orgheadline84}
\label{orgtarget31}
We present an inductive property of \(\sgF^{\ybs}\), and then show Theorem \ref{orgtarget4}.
Using this theorem, we prove Theorem \ref{orgtarget22}.

\subsection{Carry and inductive property of \texorpdfstring{$\sgF^{\ybs}$}{Fb}}
\label{sec:orgheadline78}
\label{orgtarget26}
Carry plays an important role in this section.
For \(n, h \in \Nb\),  define
\[
 r(n, h) = r^{\ybs}(n, h) = (n + h) \ominus (n \oplus h)
\]
and \(r_L(n, h) = (r(n, h))_L\).
Then
\[
 r_L(n, h) = \begin{cases}
 1 & \tif  \text{there is a carry in the \(L\)th digit in the calculation of\ } n + h \\
 & \text{\ \  in base\ } \ybs,\\
 0 & \totherwise.
 \end{cases}
\]
For example, if \(\ybs = (10,10, \ldots), n = 37\), and \(h = 65\), then \(r(n, h) = 102 \ominus 92 = 110\).
Note that 
\begin{equation}
\label{orgtarget65}
 \Chopped{n + h} = \Chopped{n} + \Chopped{h} + r_1(n, h),
\end{equation}
where $\Chopped{n} = n_{\ge 1}$.
In the above example, \(\Chopped{n + h} = \Chopped{102} = 10\) and \(\Chopped{n} + \Chopped{h} + r_1(n, h) = 3 + 6 + 1 = 10\).
For \(X, F \in \Nb^m\), let 
\[
 r(X, F) = (r(x^0, f^0), \ldots, r(x^{m - 1}, f^{m - 1}))
\]
and \(r_L(X, F) = (r_L(x^0, f^0), \ldots, r_L(x^{m - 1}, f^{m - 1}))\).
Observe that
\begin{equation}
\label{orgtarget66}
 \begin{split}
 \Carry(F) &= \set{r_1(X, F) : X \in \Nb^m} \\
 &= \set{r_1(x, F) : x \in \range{\ybs[0]}^m},
 \end{split}
\end{equation}
where \(\Carry(F)\) is defined in (\ref{orgtarget25}).
For example, if \(\ybs = (10,10, \ldots)\) and \(F = (37, 10)\), then
\(\Carry(F) = \set{0,1} \times \set{0}\).

The next result allows us to prove Theorem \ref{orgtarget4} by induction.

 \begin{lemma}
 \comment{Lem.}
\label{sec:orgheadline76}
\label{orgtarget67}
If \(F \in \cF^{\ybs}\), then \(\Chopped{F} + r \in \cF^{\Chopped{\ybs}}\) for some \(r \in \Carry(F)\).
 
\end{lemma}

\begin{proof}
 \comment{Proof.}
\label{sec:orgheadline77}
Since \(F \in \cF^{\ybs}\), it follows that \(\nsigma[\ybs](X + F) = \nsigma[\ybs](X)\) for some \(X \in \Nb^m\),
and hence that \(\nsigma[\Chopped{\ybs}](\Chopped{X + F}) = \nsigma[\Chopped{\ybs}](\Chopped{X})\).
By (\ref{orgtarget65}), \(\Chopped{X + F} = \Chopped{X} + \Chopped{F} + r_1(X, F)\), so
\[
 \nsigma[\Chopped{\ybs}](\Chopped{X}) = \nsigma[\Chopped{\ybs}](\Chopped{X + F}) = \nsigma[\Chopped{\ybs}](\Chopped{X} + \Chopped{F} + r_1(X, F)).
\]
We conclude that \(\Chopped{F} + r_1(X, F) \in \sgF^{\Chopped{\ybs}}\).

\end{proof}

\subsection{Proof of Theorem \texorpdfstring{\ref{orgtarget4}}{1.12}}
\label{sec:orgheadline79}
\label{orgtarget68}
We show (\ref{orgtarget27}) by induction on \(L\). 
Let \(\sgFD<L>\) be the left-hand side of (\ref{orgtarget27}).

First, suppose that \(L = 0\). In this case, \(\sgFD<0> = \sgFA<0>\).
Indeed, let \(F \in \sgFA<0>\). Then \(\nsigma[\ybs]<0>(F) = 0\) and \(\Chopped{F} = (0,\ldots,0)\), so 
\[
 \nsigma[\ybs](F + (0,\ldots,0)) = 0 = \nsigma[\ybs]((0,\ldots,0)).
\]
Hence \(F \in \sgFD<0>\). Conversely, let \(F \in \sgFD<0>\). Then \(\Chopped{F} = (0, \ldots, 0)\). Since \(\sgFD<0> \subseteq \sgF^{\ybs}\), there exists \(X \in \Nb^m\) with \(\nsigma[\ybs](X + F) = \nsigma[\ybs](X)\).
Because \(\nsigma[\ybs]<0>(X + F) = \nsigma[\ybs]<0>(X) \oplus \nsigma[\ybs]<0>(F)\), we see that \(\nsigma[\ybs]<0>(F) = 0\).
Hence \(F \in \sgFA<0>\). 
In particular, (\ref{orgtarget27}) holds when \(L = 0\).

Next, suppose that \(L \ge 1\).

We first verify that \(\sgFD<L> \subseteq \sgFA<L>\).
Let \(F \in \sgFD<L>\). 
Then \(\nsigma[\ybs]<0>(F) = 0\) since \(F \in \sgFD<L> \subseteq \sgF^{\ybs}\).
We prove that \(\Chopped{F} +\ccr \in \sgFA[\Chopped{\ybs}]<L - 1>\) for some \(\ccr \in \rho(F)\).
By Lemma \ref{orgtarget67}, \(\Chopped{F} + r \in \sgF^{\Chopped{\ybs}}\) for some \(r \in \rho(F)\).
To apply the induction hypothesis to \(\Chopped{F} + r\), 
we show that 
\begin{equation}
\label{orgtarget69}
 \max (\Chopped{F} + r) \le \ybsp[L][\Chopped{\ybs}] - \ybsp[L - 1][\Chopped{\ybs}].
\end{equation}
Note that \(f^i \le \ybsp[L + 1] - \ybsp[L]\) since \(F \in \sgFD<L>\).
If \(f^i < \ybsp[L + 1] - \ybsp[L] = \bnum{0, \ldots, 0, \ominus 1}\), then \(\Chopped{f}^i < \ybsp[L][\Chopped{\ybs}] - \ybsp[L - 1][\Chopped{\ybs}]\), so
\(\Chopped{f}^i + r^i \le \Chopped{f}^i + 1 \le \ybsp[L][\Chopped{\ybs}] - \ybsp[L - 1][\Chopped{\ybs}]\).
Suppose that \(f^i = \ybsp[L + 1] - \ybsp[L]\).
Since \(L \ge 1\), it follows that \(f^i_0 = 0\), and hence that \(r^i\) must be zero.
This implies that \(\Chopped{f}^i + r^i = \ybsp[L][\Chopped{\ybs}] - \ybsp[L - 1][\Chopped{\ybs}]\).
Therefore (\ref{orgtarget69}) holds, and hence \(\Chopped{F} + r \in \sgFD[\Chopped{\ybs}]<L - 1>\).
We can now apply the induction hypothesis to \(\Chopped{F} + r\), so \(\Chopped{F} + r \in \sgFA[\Chopped{\ybs}]<L - 1>\). 
Therefore \(F \in \sgFA<L>\).

We next show that \(\sgFA<L> \subseteq \sgF^{\ybs}\).
For \(F \in \sgFA<L>\), we construct \(X \in \Nb^m\) such that \(\nsigma[\ybs](X + F) = \nsigma[\ybs](X)\).
Since \(F \in \sgF_L^{\ybs}\), there exists \(r \in \Carry(F)\) such that \(\Chopped{F} + \ccr \in \sgFA[\Chopped{\ybs}]<L - 1>\).
By the induction hypothesis, \(\Chopped{F} + \ccr \in \sgF^{\Chopped{\ybs}}\), so
\(\nsigma[\Chopped{\ybs}](\Chopped{X} + \Chopped{F} + \ccr) = \nsigma[\Chopped{\ybs}](\Chopped{X})\) for some \(\Chopped{X} \in \Nb[\Chopped{\ybs}]^m\).
From (\ref{orgtarget66}), we can choose \(x \in \range{\ybs[0]}^m\) so that \(\ccr = r_1(x, F)\). Let \(X = x + \ybs[0] \Chopped{X} \in \Nb^m\).
We show that \(\nsigma[\ybs](X + F) = \nsigma[\ybs](X)\).
Since \(F \in \sgFA<L>\), we see that \(\nsigma[\ybs]<0>(F)= 0\). Hence
\[
 \nsigma[\ybs]<0>(X + F) =  \nsigma[\ybs]<0>(X) \oplus \nsigma[\ybs]<0>(F) = \nsigma[\ybs]<0>(X).
\]
Moreover, since \(r_1(X, F) = r_1(x, F) = r\), it follows that \(\Chopped{X + F} = \Chopped{X} + \Chopped{F} + \ccr\), and hence that
\begin{align*}
 (\nsigma[\ybs](X + F))_{\ge 1} &= \nsigma[\Chopped{\ybs}](\Chopped{X + F})  \\
 &= \nsigma[\Chopped{\ybs}](\Chopped{X} + \Chopped{F} + \ccr)  \\
 &= \nsigma[\Chopped{\ybs}](\Chopped{X}) = (\nsigma[\ybs](X))_{\ge 1}.
\end{align*}
Therefore \(\nsigma[\ybs](X + F) = \nsigma[\ybs](X)\).

\qed

\subsection{Proof of Theorem \texorpdfstring{\ref{orgtarget22}}{1.14}}
\label{sec:orgheadline83}
\label{orgtarget70}
(\(\Rightarrow\)). We first show that if \(\ybs \neq (\ybs[0], 2,\ldots)\), then
\(\sgCmax[\ybs] \supsetneq \sgCord[\ybs]\).
Let
\[
 N = \min \set{L \in \NN : L \ge 1,\ \ybs[L] \ge 3},
\]
and consider 
\[
 \ccC = (\bnums{\ccc^0_0, \underbrace{0, \ldots, 0}_N, 1}, \bnums{\ominus \ccc^0_0}, 0, \ldots, 0) = (\ccc^0_0 + \ybsp[N + 1], \ominus \ccc^0_0, 0, \ldots, 0) \in \Nb^m.
\]
It suffices to show that \(\ccC \in \sgCmax[\ybs]\) since
\(\ccC \not \in \sgCord[\ybs]\) if \(\ccc^0_0 \neq 0\).
Lemma \ref{orgtarget67} implies that \(\ccC \in \sgCmax[\ybs]\) if
\begin{equation}
\label{orgtarget71}
\Chopped{\ccC} + r  \in \sgCmax[\Chopped{\ybs}] \tforevery r \in \Carry(\ccC).
\end{equation}
We show (\ref{orgtarget71}) by induction on \(N\).
Let \(\ccr \in \Carry(\ccC)\). If \(\nsigma[\Chopped{\ybs}]<0>(\Chopped{\ccC} + \ccr) \neq 0\), then 
\(\Chopped{\ccC} + \ccr \in \sgCmax[\Chopped{\ybs}]\). Hence we may assume that \(\nsigma[\Chopped{\ybs}]<0>(\Chopped{\ccC} + \ccr) = 0\).

Suppose that \(N = 1\).
Then \(\ccC = (\bnum{\ccc^0_0, 0, 1}, \bnum{\ominus \ccc^0_0}, 0, \ldots, 0)\).
We see that \(\ccr = (0,\ldots,0)\) because \(\ccr \in \Carry(\ccC) \subseteq \set{0, 1} \times \set{0, 1} \times \set{0} \times \cdots \times \set{0}\) and \(\ybs[1] \ge 3\).
Hence \(\Chopped{\ccC} + \ccr = \Chopped{\ccC} = (\bnum{0,1}, 0, \ldots, 0)  \in \sgCmax[\Chopped{\ybs}]\).

Suppose that \(N > 1\).
Since \(\ybs[1] = 2\) and \(\nsigma[\ybs]<0>(\Chopped{\ccC} + \ccr) = 0\), we see that \(\Chopped{\ccC} + \ccr = (\epsilon + \Chopped{\ybs}^N,   \epsilon, 0, \ldots, 0)\) for some \(\epsilon \in \set{0, 1}\).
By the induction hypothesis, \(\Chopped{\ccC} + \ccr \in \sgCmax[\Chopped{\ybs}]\). Therefore \(\ccC \in \sgCmax[\ybs]\).

(\(\Leftarrow\)). We next show that 
if \(\ybs = (\ybs[0], 2,\ldots)\), then
\(\sgCmax[\ybs] = \sgCord[\ybs]\).
Since \(\sgCord[\ybs] \subseteq \sgCmax[\ybs]\),
it suffices to show that \(\Nb^m \setminus \sgCord[\ybs] \subseteq \Nb^m \setminus \sgCmax[\ybs] (= \sgF^{\ybs})\).
Set \(\sgG^{\ybs} = \Nb^m \setminus \sgCord[\ybs]\). Then
\[
 \sgG^{\ybs}  = \Set{\ccG \in \Nb^m : \ord_{\ybs}\left(\sum g^i\right) > \mord_{\ybs}(\ccG)} \cup \set{(0, \ldots, 0)}.
\]
Let \(\ccG \in \sgG^{\ybs}\).
If \(\ccG = (0,\ldots,0)\), then \(\ccG \in \sgFA\).
Suppose that \(\ccG \neq (0,\ldots, 0)\), and let
\[
 \ccM = \mord_{\ybs}(G).
\]
Note that \(\nsigma[\ybs]<0>(G) = 0\) since \(\ord_{\ybs}(\sum \ccg^i) > \ccM\).
We show that \(G \in \sgF^{\ybs}\) by induction on \(\max \ccG\).
If \(\max \ccG \le \ybs[0] - 1\), then \(\ccG  \in \sgFA<0> \subset \sgF^{\ybs}\).
Suppose that \(\max \ccG \ge \ybs[0]\).
We divide into two cases.

 \begin{mycase}[\(\ccM > 0\)]
 \comment{Case. [\(\ccM > 0\)]}
\label{sec:orgheadline80}
We show that \(\Chopped{\ccG} \in \sgF^{\Chopped{\ybs}}\).
Since \(\ccM > 0\), it follows that \((\ccg^0_0, \ldots, \ccg^{m - 1}_0) = (0,\ldots, 0)\), and hence that
\[
\ord_{\Chopped{\ybs}}\left(\sum \Chopped{\ccg}^i\right) =  \ord_{\ybs}\left(\sum \ccg^i\right) - 1 > \ccM - 1 = \mord_{\Chopped{\ybs}}(\Chopped{\ccG}).
\]
This implies that \(\Chopped{\ccG} \in \sgG^{\Chopped{\ybs}}\).
Since \(\max \Chopped{\ccG} < \max \ccG\), it follows from the induction hypothesis that \(\Chopped{\ccG} \in \sgF^{\Chopped{\ybs}}\). Hence \(\ccG \in \sgF^{\ybs}\) by Theorem \ref{orgtarget4}.
 
\end{mycase}

 \begin{mycase}[\(\ccM = 0\)]
 \comment{Case. [\(\ccM = 0\)]}
\label{sec:orgheadline81}
Since \((\ccg^0_0, \ldots, \ccg^{m - 1}_0) \neq (0,\ldots, 0)\) and \(\ord_{\ybs}\left(\sum \ccg^i\right) \ge 1\),
the number of \(i\) with \(g^i_0 \neq 0\) is at least 2.
Hence we can choose  \(\ccr \in \Carry(G)\) so that the number of \(i\) with \((\Chopped{g}^i + r^i)_0 \neq 0\) is even and greater than 0. 
Then \(\ord_{\ybs}(\sum \Chopped{\ccg}^i + \ccr^i) \ge 1 > 0 = \mord_{\ybs}(\Chopped{G} + \ccr)\),
and hence \(\Chopped{\ccG} + \ccr \in \sgG^{\Chopped{\ybs}}\).
To apply the induction hypothesis to \(\Chopped{\ccG} + r\), we show that \(\max (\Chopped{\ccG} + r) < \max \ccG\).
If \(\ccg^i < \ybs[0]\), then \(\Chopped{\ccg}^i + r^i = r^i < \ybs[0] \le \max \ccG\).
If \(\ccg^i \ge \ybs[0]\), then \(\Chopped{\ccg}^i + r^i < \ybs[0] \Chopped{\ccg}^i + g^i_0 = \ccg^i\).
Hence \(\max (\Chopped{\ccG} + r) < \max \ccG\).
It follows from the induction hypothesis that \(\Chopped{\ccG} + r \in \sgF^{\Chopped{\ybs}}\).
Therefore \(\ccG \in \sgF^{\ybs}\) by Theorem \ref{orgtarget4}.

\qed
 
\end{mycase}

\comment{bibliography}
\label{sec:orgheadline82}
\bibliographystyle{abbrv}

\end{document}